\documentclass[12pt, reqno]{amsart}
\usepackage{}
\usepackage{stmaryrd}
\usepackage{amsmath}
\usepackage{amsfonts}
\usepackage{amssymb}
\usepackage[all,cmtip]{xy}           
\usepackage{xspace}
\usepackage{bbding}
\usepackage{txfonts}
\usepackage[shortlabels]{enumitem}
\usepackage{cite}
\usepackage{longtable}
\usepackage{makecell}
\usepackage{tikz}
\usepackage{titletoc}
\usepackage{tikz-cd}
\usepackage{ifpdf}
\ifpdf
 \usepackage[colorlinks,final,backref=page,hyperindex]{hyperref}
\else
 \usepackage[colorlinks,final,backref=page,hyperindex,hypertex]{hyperref}
\fi
\usepackage[active]{srcltx}

\makeatletter

\begin{document}
\newcommand {\emptycomment}[1]{} 

\baselineskip=15pt
\newcommand{\nc}{\newcommand}
\newcommand{\delete}[1]{}
\nc{\mfootnote}[1]{\footnote{#1}} 
\nc{\todo}[1]{\tred{To do:} #1}

\nc{\mlabel}[1]{\label{#1}}  
\nc{\mcite}[1]{\cite{#1}}  
\nc{\mref}[1]{\ref{#1}}  
\nc{\meqref}[1]{\eqref{#1}} 
\nc{\mbibitem}[1]{\bibitem{#1}} 

\delete{
\nc{\mlabel}[1]{\label{#1}  
{\hfill \hspace{1cm}{\bf{{\ }\hfill(#1)}}}}
\nc{\mcite}[1]{\cite{#1}{{\bf{{\ }(#1)}}}}  
\nc{\mref}[1]{\ref{#1}{{\bf{{\ }(#1)}}}}  
\nc{\meqref}[1]{\eqref{#1}{{\bf{{\ }(#1)}}}} 
\nc{\mbibitem}[1]{\bibitem[\bf #1]{#1}} 
}

\newcommand {\comment}[1]{{\marginpar{*}\scriptsize\textbf{Comments:} #1}}
\nc{\mrm}[1]{{\rm #1}}
\nc{\id}{\mrm{id}}  \nc{\Id}{\mrm{Id}}
\nc{\admset}{\{\pm x\}\cup (-x+K^{\times}) \cup K^{\times} x^{-1}}

\def\a{\alpha}
\def\admt{admissible to~}
\def\ad{associative D-}
\def\asi{ASI~}
\def\aybe{aYBe~}
\def\b{\beta}
\def\bd{\boxdot}
\def\bbf{\overline{f}}
\def\bF{\overline{F}}
\def\bg{\overline{g}}
\def\bG{\overline{G}}
\def\bT{\overline{T}}
\def\bt{\overline{t}}
\def\bR{\overline{R}}
\def\br{\overline{r}}
\def\bu{\overline{u}}
\def\bU{\overline{U}}
\def\bw{\overline{w}}
\def\bW{\overline{W}}
\def\btl{\blacktriangleright}
\def\btr{\blacktriangleleft}
\def\calo{\mathcal{O}}
\def\ci{\circ}
\def\d{\delta}
\def\dd{\diamondsuit}
\def\D{\Delta}
\def\frakB{\mathfrak{B}}
\def\G{\Gamma}
\def\g{\gamma}
\def\k{\kappa}
\def\l{\lambda}
\def\lh{\leftharpoonup}
\def\lr{\longrightarrow}
\def\N{Nijenhuis~}
\def\o{\otimes}
\def\om{\omega}
\def\opa{\cdot_{A}}
\def\opb{\cdot_{B}}
\def\p{\psi}
\def\sadm{$S$-admissible~}
\def\r{\rho}
\def\ra{\rightarrow}
\def\rh{\rightharpoonup}
\def\rr{r^{\#}}
\def\s{\sigma}
\def\st{\star}
\def\ti{\times}
\def\tl{\triangleright}
\def\tr{\triangleleft}
\def\v{\varepsilon}
\def\vp{\varphi}
\def\vth{\vartheta}

\newtheorem{thm}{Theorem}[section]
\newtheorem{lem}[thm]{Lemma}
\newtheorem{cor}[thm]{Corollary}
\newtheorem{pro}[thm]{Proposition}
\theoremstyle{definition}
\newtheorem{defi}[thm]{Definition}
\newtheorem{ex}[thm]{Example}
\newtheorem{rmk}[thm]{Remark}
\newtheorem{pdef}[thm]{Proposition-Definition}
\newtheorem{condition}[thm]{Condition}
\newtheorem{question}[thm]{Question}
\renewcommand{\labelenumi}{{\rm(\alph{enumi})}}
\renewcommand{\theenumi}{\alph{enumi}}

\title{Nijenhuis operators and mock-Lie bialgebras}

 \author[Ma]{Tianshui Ma }
 \address{School of Mathematics and Information Science, Henan Normal University, Xinxiang 453007, China}
         \email{matianshui@htu.edu.cn}

 \author[Mabrouk]{Sami Mabrouk}
 \address{University of Gafsa, Faculty of Sciences, 2112 Gafsa, Tunisia}

         \email{mabrouksami00@yahoo.fr, sami.mabrouk@fsgf.u-gafsa.tn}

 \author[Makhlouf]{Abdenacer Makhlouf\textsuperscript{*}}
 \address{Universit{\'e} de Haute Alsace, IRIMAS-D\'epartement  de Math{\'e}matiques,18  rue des Fr{\`e}res Lumi{\`e}re F-68093 Mulhouse, France}

         \email{abdenacer.makhlouf@uha.fr}

 \author[Song]{Feiyan Song}
 \address{School of Mathematics and Information Science, Henan Normal University, Xinxiang 453007, China}
         \email{songfeiyan@stu.htu.edu.cn}

 \thanks{\textsuperscript{*}Corresponding author}

\date{\today}

 \begin{abstract} A Nijenhuis mock-Lie algebra is a mock-Lie algebra equipped with a Nijenhuis operator. The purpose of this paper is  to extend the well-known results about Nijenhuis mock-Lie algebras to the realm of mock-Lie bialgebras. It aims to  characterize Nijenhuis mock-Lie bialgebras by generalizing the concepts of matched pairs and Manin triples of mock-Lie algebras to the context of Nijenhuis mock-Lie algebras. Moreover, we discuss formal deformation theory and explore  infinitesimal  formal deformations of Nijenhuis mock-Lie algebras, demonstrating that the associated cohomology corresponds to a deformation cohomology. Moreover, we define  abelian extensions of Nijenhuis mock-Lie algebras and show that equivalence classes of such extensions are linked to cohomology groups. The coboundary case leads to the introduction of an admissible mock-Lie-Yang-Baxter equation (mLYBe) in Nijenhuis mock-Lie algebras, for which the antisymmetric solutions give rise to Nijenhuis mock-Lie bialgebras. Furthermore, the notion of $\mathcal O$-operator on Nijenhuis mock-Lie algebras is  introduced and connected to mock-Lie-Yang-Baxter equation.\end{abstract} 
 \subjclass[2020]{
16T10,   
16T25,   
17B38,  
16W99,  
17B62.   
}

\keywords{Nijenhuis operator, mock-Lie algebra, mock-Lie bialgebra, deformation, abelian extension, coboundary, $\mathcal{O}$-operator.}

\maketitle


 \tableofcontents


\allowdisplaybreaks

\section{Introduction}
\subsection{Mock-Lie algebras}

In the world of non-associative algebras, {\bf mock-Lie algebras} 
are rather special objects. A mock-Lie algebra is a vector space $\mathcal{A}$
together with a commutative bilinear map $[~,~] : \mathcal{A} \times \mathcal{A} \to \mathcal{A}$ such
that 
 \begin{eqnarray}
 &[x, y]=[y, x],&\label{eq:1-1}\\
 &[x, [y, z]]+[y, [z, x]]+[z, [x, y]]=0,\mlabel{eq:i}&
 \end{eqnarray}
 where $x, y, z\in \mathcal{A}$.
A detailed history of mock-Lie algebras and  several conjectures are presented in \cite{zu2}. This class of algebras was first introduced in \cite{Zh}. Since then, they have been independently studied in various works \cite{BB0, BB, BB1, BB2, BF, Braiek, Camacho, GK, Ok, wa, Buse, Zh2} under different names, including Lie-Jordan algebras, Jordan algebras of nilindex 3, pathological algebras, mock-Lie algebras, or Jacobi-Jordan algebras. In this paper, we will use the term mock-Lie algebra.

\subsection{Deformations, Nijenhuis operators and abelian extensions} 

The theory of deformations of algebraic structures was first introduced by Gerstenhaber for associative algebras \cite{Gerst}, and later extended to Lie algebras by Nijenhuis and Richardson \cite{NijenhuisRichardson,NR,NR1}. The core idea is to take a non-associative algebra 
$(A,\mu)$ and define a new multiplication 
$\mu_t$  on the formal space 
$A[[t]]$ by incorporating additional terms in powers of the formal parameter $t$ into $\mu$. The new multiplication $\mu_t$ will then be of the form:
\begin{equation*}
    \mu_t=\mu +\sum_{i \geq 1}\mu_it^i,
\end{equation*}
where $\mu_i:A\otimes A \rightarrow A$ are bilinear maps. Depending on the structure we wish to have on $A[[t]]$, certain conditions are necessary on the $\mu_i$ applications. An infinitesimal  deformation $(t^2=0)$ is said to be trivial if there exists a linear map $N : A\to A$ such
that for all $t$, $T_t= id + tN $ satisfies $T_t(\mu_t(a,b))=\mu(T_t(a),T_t(b))$. Thus $N$ satisfying, for all $a,b\in A$, $$\mu(N(a), N(y))=N(\mu(N( a), b))+N(\mu(a, N(b))-N(\mu(a, b))).$$
Such a  linear map $N$ is called Nijenhuis operator.  Dorfman \cite{Dorfman1993} examined Nijenhuis operators through the deformation of Lie algebras. Additionally, these operators are significant in the context of the integrability of nonlinear evolution equations \cite{Dorfman1993}. Dorfman's introduction of Dirac structures provided new insights into existing Nijenhuis frameworks. In 2004, Gallardo and Nunes \cite{CN04} introduced the concept of Dirac Nijenhuis structures, while Longguang and Baokang \cite{LB04} independently developed Dirac Nijenhuis manifolds that same year. In 2011, Kosmann-Schwarzbach \cite{kos} investigated Dirac Nijenhuis structures within Courant algebroids. A concept of a Nijenhuis operator on a 3-Lie algebra was proposed in \cite{zhang} to explore first-order deformations of such algebras. In \cite{Wang}, the authors defined a Nijenhuis operator for pre-Lie algebras that leads to a trivial deformation of the algebra. Das and Sen \cite{DS} analyzed Nijenhuis operators on Hom-Lie algebras from a cohomological perspective. For further research on Nijenhuis operators across various algebraic structures, see \cite{LSZB, HCM}.

Non-abelian extensions were first introduced by Eilenberg and Maclane \cite{nonC3}, leading to the development of low-dimensional non-abelian cohomology groups. These extensions serve as valuable mathematical tools for examining the underlying structures. Among the various types of extensions as central, abelian, and non-abelian, the non-abelian extension is particularly general. As a result, many studies have concentrated on abelian extensions of diverse algebraic structures, including Lie (super)algebras, Leibniz algebras, Lie $2$-algebras, Lie Yamaguti algebras, Rota-Baxter groups, Rota-Baxter Lie algebras, and Rota-Baxter Leibniz algebras (see \cite{nab6, DR, GKM, nab13, nab10, nonC2}). For the cohomology and deformations of mock-Lie algebras, we refer to \cite{BB2}.

\subsection{Bialgebras} A bialgebra structure is obtained from a corresponding set of comultiplication together with a set of compatibility conditions between the multiplication and the comultiplication. For example, for a finite-dimensional vector space $V$ with a given algebraic structure, this can be achieved by equipping the dual space $V^*$ with the same algebraic
structure and a set of compatibility conditions between the structures on $V$ and those on $V^*$.
One important reason for the usefulness of the Lie bialgebra is that it has a coboundary theory, which
leads to the construction of Lie bialgebras from solutions of the classical Yang-Baxter equation.
The origin of the Yang-Baxter-equations is purely physics. They were first introduced by Baxter, and Yang in \cite{R.J1,R.J2, C.N}.
The notion of mock-Lie bialgebra was introduced in \cite{BCHM} as an equivalent to a Manin triple of mock-Lie algebras (see also \cite{Karima, Ismail}). The authors also studied a special case called coboundary mock-Lie
bialgebra which leads to the introduction  of the mock-Lie Yang-Baxter equation on a mock-Lie algebra as an analogue of the classical Yang-Baxter equation on a Lie algebra.
\subsection{Layout of the paper}
The paper is organized as follows: In Section \ref{Sec2},  we introduce the notion of  Nijenhuis mock-Lie algebra and provide some relevant  definitions and properties. Moreover, we revisit mock-Lie bialgebras and discuss classical mock-Lie Yang-Baxter equation and (dual)quasitriangular mock-Lie bialgebras. Furthermore, we recall the construction Symplectic mock-Lie algebras from dual quasitriangular mock-Lie bialgebras. Then, we show that they lead to Nijenhuis mock-Lie algebras. Section \ref{Sec3} is dedicated to Deformations and extensions of Nijenhuis mock-Lie algebras. We study infinitesimal deformations and abelian extensions of  Nijenhuis mock-Lie algebras. In Section \ref{Sec4}, we introduce and study Matched pairs and Manin triples of Nijenhuis mock-Lie algebras. Therefore, we show their relationship with Nijenhuis mock-Lie bialgebras. The class of Coboundary Nijenhuis mock-Lie bialgebras is also considered through $S$-adjoint-admissible Nijenhuis mock-Lie algebras,  $S$-admissible mock-Lie-Yang-Baxter equation and $\mathcal O$-operators.

 \noindent{\bf Notations:} Throughout this paper, all vector spaces, tensor products, and linear homomorphisms are over a field
$K$ of characteristic zero. We denote by $\id_M$ the identity map from $M$ to $M$.

\section{From mock-Lie bialgebras to Nijenhuis mock-Lie algebras}\label{Sec2}
In this section, We introduce the concept of Nijenhuis mock-Lie algebras and present key definitions and properties. Additionally, we revisit mock-Lie bialgebras (see \cite{BCHM}), exploring the classical mock-Lie Yang-Baxter equation and (dual) quasitriangular mock-Lie bialgebras. Furthermore, we recall the construction of symplectic mock-Lie algebras derived from dual quasitriangular mock-Lie bialgebras. We demonstrate how these constructions give rise to Nijenhuis mock-Lie algebras.
\subsection{Nijenhuis mock-Lie algebras}

 \begin{defi}[\cite{BCHM}]\label{de:ho} Let $(\mathcal{A},[~,~])$ be a mock-Lie algebra, $V$ be a vector space and $\rho : \mathcal{A}\longrightarrow  End(V)$ be a linear map. The pair $(V,\rho)$ is called a {\bf representation} of $\mathcal{A}$ if for all $x,y \in \mathcal{A}, v\in V$
 \begin{eqnarray}\label{eq:kp}
 \rho([x,y])(v)=-\rho(x)(\rho(y)v)-\rho(y)(\rho(x)v).
 \end{eqnarray}
 \end{defi}

 \begin{rmk}[\cite{BCHM}]\label{rmk:gho} Let $(\mathcal{A},[~,~])$ be a mock-Lie algebra, $V$ be a vector space and $\rho : \mathcal{A}\longrightarrow End(V)$ be a linear map. The pair $(V,\rho)$ is a representation of $\mathcal{A}$ if and only if the direct sum $\mathcal{A}\oplus V$ of vector spaces is turned into a mock-Lie algebra by defining multiplication on $\mathcal{A}\oplus V$ by
 \begin{eqnarray}\label{eq:m}
 [x+u,y+v]_\star=[x,y]+\rho(x)v+\rho(y)u, \forall x,y\in \mathcal{A}, u,v\in V.
 \end{eqnarray}
 We denote this mock-Lie algebra by $\mathcal{A}{\ltimes_\rho} V$.
 \end{rmk}

\delete{\begin{proof}
 For all $ x,y,z \in \mathcal{A}$   and  $u,v,w \in V$, we have
\begin{eqnarray*}
&&[x+u,[y+v,z+w]_\star]_\star+[y+v,[z+w,x+u]_\star]_\star+[z+w,[x+u,y+v]_\star]_\star\\
&\stackrel {(\ref{eq:m})}{=}&[x+u,[y,z]+\rho(y)w+\rho(z)v]_\star+[y+v,[z,x]+\rho(z)u+\rho(x)w]_\star+[z+w,[x,y]+\rho(x)v+\rho(y)u]_\star\\,      
&=&[x,[y,z]]+\rho(x)(\rho(y)w+\rho(z)v)+\rho([y,z])u+[y,[z,x]]+\rho(y)(\rho(z)u+\rho(x)w)+\rho([z,x])v\\
&&+[z,[x,y]]+\rho(z)(\rho(x)v+\rho(y)u)+\rho([x,y])w
\end{eqnarray*}
Hence, $(V,\rho)$ is a module over mock-Lie algebra $\mathcal{A}$ if and only if $\mathcal{A}\oplus V$ is a mock-Lie algebra.
\end{proof}
}

 \begin{lem}\label{lem:bgu} Let $\mu:\mathcal{A}\otimes \mathcal{A}\longrightarrow \mathcal{A}$, $\rho: \mathcal{A}\longrightarrow End(V)$ and $m:\mathcal{A}\otimes \mathcal{A}\longrightarrow \mathcal{A}$, $\varphi: \mathcal{A}\longrightarrow End(V)$ be linear maps (we write $\mu(x\otimes y)=[x,y]$, $m(x\otimes y)=[x,y]_\cdot$), $s, t$ be parameters. Consider the following new multiplication and linear map:
 \begin{eqnarray*}
 [x,y]_\circ=s[x,y]+t[x,y]_\cdot,\quad \xi(x)=s\rho(x)+t\varphi(x),~\forall~ x,y\in \mathcal{A}.
 \end{eqnarray*}
 Then
 \begin{enumerate}[(1)]
   \item $(\mathcal{A},[~,~]_\circ)$ is a mock-Lie algebra if and only if $(\mathcal{A}, [~,~]_\cdot)$ is a mock-Lie algebra and for all $x,y,z\in \mathcal{A}$,
   \begin{eqnarray}\label{eq:p4}
   [x,[y,z]]_\cdot+[x,[y,z]_\cdot]+[y,[z,x]]_\cdot+[y,[z,x]_\cdot]+[z,[x,y]]_\cdot+[z,[x,y]_\cdot]=0.
   \end{eqnarray}
   \item $(V,\xi)$ is a representation of $(\mathcal{A},[~,~]_\circ)$ if and only if $(V,\vp)$ is a representation of $(\mathcal{A},[~,~]_\cdot)$ and for all $x,y \in \mathcal{A}$,
   \begin{eqnarray}\label{eq:rft}
   \rho([x,y]_\cdot)+\varphi([x,y])=-\varphi(x)\rho(y)-\rho(x)\varphi(y)-\varphi(y)\rho(x)-\rho(y)\varphi(x).
   \end{eqnarray}
 \end{enumerate}
 \end{lem}

 \begin{proof}
 Straightforward.
 \end{proof}

 \begin{thm}\label{thm:rl} Let $(V,\rho)$ be a representation $(\mathcal{A},[~,~])$, $(V,\xi)$ be a representation of $(\mathcal{A},[~,~]_\circ)$, $N\in End(\mathcal{A})$ and $\alpha\in End(V)$ be two linear maps. Then $(s\;\id_\mathcal{A}+t N, s\;\id_V+t \alpha)$ is a homomorphism from $(V, \xi)$ of $(\mathcal{A}, [~,~]_\ci)$ to $(V, \rho)$ of $(\mathcal{A}, [~,~])$ if and only if, for all $x,y\in \mathcal{A}$,
 \begin{eqnarray}
 &[x,y]_\cdot=[N(x),y]+[x,N(y)]-N[x,y],&\label{eq:vb}\\
 &N([x,y]_\cdot)=[N(x),N(y)],&\label{eq:yq}\\
 &\varphi(x)=\rho(N(x))+\rho(x) \alpha-\alpha \rho(x),&\label{eq:ljm}\\
 &\rho(N(x)) \alpha=\alpha \varphi(x).&\label{eq:fxm}
 \end{eqnarray}
 \end{thm}

 \begin{proof}
 Straightforward.
 \end{proof}

\delete{
\begin{proof}
By Eq.$(\mref{eq:pkh})$, for all $x,y\in \mathcal{A}$, we have
\begin{eqnarray*}
&&(sId_\mathcal{A}+tN)[x,y]_\circ =(sId_\mathcal{A}+tN)(s[x,y]+t[x,y]_\cdot)\\
&&=s^2[x,y]+st[x,y]_\cdot+tsN([x,y])+t^2N([x,y]_\cdot)\\
&&=[(sId_\mathcal{A}+tN)(x),(sId_\mathcal{A}+tN)(y)]=[sx+tN(x),sy+tN(y)]\\
&&=s^2[x,y]+st[x,N(y)]+ts[N(x),y]+t^2[N(x),N(y)].
\end{eqnarray*}
From the above equation, we can deduce Eq.$(\mref{eq:vb})$ and $(\mref{eq:yq})$ hold. Similarly, By Eq.$(\mref{eq:rty})$, we have
\begin{eqnarray*}
&&(sId_V+t\alpha)\circ \xi(x)(v)=(sId_V+t\alpha)\circ (s\rho(x)v+t\varphi(x)v)\\
&&=s^2\rho(x)v+st\varphi(x)v+ts\alpha(\rho(x)v)+t^2\alpha(\varphi(x)v)\\
&&=\xi(sId_\mathcal{A}+tN)\circ (sId_V+t\alpha)(v)=\xi(sId_\mathcal{A}+tN)(sv+t\alpha(v))\\
&&=s^2\rho(x)v+st\rho(x)(\alpha(v))+ts\rho(N(x))(v)+t^2\rho(N(x))\alpha(v).
\end{eqnarray*}
\end{proof}
}

 By Eqs.(\mref{eq:vb}) and (\mref{eq:yq}) we have:
 \begin{defi}\label{de:ai} A {\bf \N mock-Lie algebra} is a triple $(\mathcal{A}, [~,~], N)$, where $(\mathcal{A}, [~,~])$ is a mock-Lie algebra and $N:\mathcal{A}\lr \mathcal{A}$ is a \N operator, i.e.,
 \begin{eqnarray}\label{eq:bja}
 [N(x), N(y)]+N^2([x, y])=N([N(x), y])+N([x, N(y)]),\quad \forall~x, y\in \mathcal{A}.
 \end{eqnarray}
 \end{defi}

 By Eqs.(\mref{eq:ljm}) and (\mref{eq:fxm}), one can obtain:
 \begin{defi}\label{de:a} Let $(\mathcal{A}, [~,~], N)$ be a Nijenhuis mock-Lie algebra. A {\bf representation} of $(\mathcal{A},[~,~],N)$ is a triple $(V,\rho,\alpha)$ where $(V,\rho)$ is a representation of $(\mathcal{A},[~,~])$ and $\alpha: V\longrightarrow V$ is a linear map such that, for all $x \in \mathcal{A}, v\in V$,
 \begin{eqnarray}\label{eq:1}
\rho(N(x))\alpha(v)+\alpha^2(\rho(x)v)=\alpha(\rho(N(x))v)+\alpha(\rho(x)\alpha (v)).
 \end{eqnarray}

 A linear map $\phi:V_1\longrightarrow V_2$ is called a {\bf homomorphism from $(V_1,\rho_1,\alpha_1)$ to $(V_2,\rho_2,\alpha_2)$} if, for all $x\in \mathcal{A}, v\in V_1$,
 \begin{eqnarray*}
 \phi(\rho_1 (x)v)=\rho_2(x)\phi(v), \quad \phi(\alpha_1 (v))=\alpha_2(\phi(v)).
 \end{eqnarray*}
 If $\phi$ is bijective, then we say that the representations $(V_1,\rho_1,\alpha_1)$ and $(V_2,\rho_2,\alpha_2)$ are {\bf equivalent}.
 \end{defi}

 \begin{ex}\label{ex:pl}
 Let $(\mathcal{A},[~,~],N)$ be a  Nijenhuis mock-Lie algebra, Then $(\mathcal{A},ad,N)$ is a representation of $(\mathcal{A},[~,~],N)$, called the {\bf adjoint representation} of $(\mathcal{A},[~,~],N)$, where $ad_x(y)=[x,y], \forall~x, y\in \mathcal{A}$.
 \end{ex}

 \begin{pro}\label{pro:B}
 Let $(V,\rho ,\alpha)$ be a representation of a Nijenhuis mock-Lie algebra $(\mathcal{A},[~,~],N)$. Define the bracket product on $\mathcal{A}\oplus V$ by Eq.(\mref{eq:m}) and a linear map $N+\alpha:\mathcal{A}\oplus V \longrightarrow \mathcal{A}\oplus V$ by
 \begin{eqnarray}\label{eq:d}
 (N+\alpha)(x+u):=N(x)+\alpha(u).
 \end{eqnarray}
 Then $(\mathcal{A}\oplus V, [~,~]_\star, N+\alpha)$ is a Nijenhuis mock-Lie algebra denoted by $(\mathcal{A}{\ltimes_\rho} V,N+\alpha)$.
 \end{pro}

 \begin{proof} By Remark \ref{rmk:gho}, $\mathcal{A}\oplus V$ is a mock-Lie algebra. For all $x, y\in \mathcal{A}$ and $u, v\in V$, we calculate
 \begin{eqnarray*}
 &&\hspace{-16mm}[(N+\alpha)(x+u),(N+\alpha)(y+v)]_\star+(N+\alpha)^2[x+u,y+v]_\star
 -(N+\alpha)[(N+\alpha)(x+u),y+v]_\star\\
 &&-(N+\alpha)[x+u,(N+\alpha)(y+v)]_\star\\
 \hspace{-3mm}&\stackrel{(\ref{eq:m})(\ref{eq:d})}{=}&\hspace{-3mm}[N(x),N(y)]+\rho(N(x))\alpha(v)+\rho(N(y))\alpha(u)+(N+\alpha)(N[x,y]+\alpha(\rho(x)v)+\alpha(\rho(y)u)\\
 \hspace{-3mm}&&\hspace{-5mm}-(N+\alpha)([N(x),y]+\rho(N(x))(v)+\rho(y)\alpha(u))-(N+\alpha)([x,N(y)]+\rho(x)\alpha(v)+(\rho(N(y)))u)\\
 \hspace{-3mm}&\stackrel{(\ref{eq:d})}{=}&\hspace{-3mm}[N(x),N(y)]+\rho(N(x))\alpha(v)+\rho(N(y))\alpha(u)+N^2([x,y])+\alpha^2(\rho(x)v)+\alpha^2(\rho(y)u)\\
 \hspace{-3mm}&&\hspace{-3mm}-N[N(x),y]-\alpha(\rho(N(x))v)-\alpha(\rho (y)\alpha(u))-N([x,N(y)])-\alpha(\rho(x)\alpha(v))-\alpha(\rho(N(y))u))\\
 \hspace{-3mm}&\stackrel {(\ref{eq:bja})(\ref{eq:1})}{=}&\hspace{-3mm}0,
 \end{eqnarray*}
 finishing the proof.
 \end{proof}

 Let $(\mathcal{A},[~,~])$ be a mock-Lie algebra and $(V,\rho)$ be a representation of $(\mathcal{A},[~,~])$. Then by \cite{BCHM}, $(V^*,\rho^*)$ is a representation of $(\mathcal{A},[~,~])$, where $\rho^*:\mathcal{A}\longrightarrow End(V^*)$ is given by
 \begin{eqnarray}\label{eq:sr}
 \langle{\rho^*}(x)u^*,v\rangle=\langle{\rho(x)}v,u^*\rangle,\quad x\in \mathcal{A}, u^*\in V^*, v\in V.
 \end{eqnarray}

 \begin{lem}\label{lem:HO} Let $(\mathcal{A},[~,~],N)$ be a Nijenhuis mock-Lie algebra, $(V,\rho)$ be a representation of $(\mathcal{A},[~,~])$ and $\beta: V\longrightarrow V$ be a linear map. Then $(V^*,\rho^*,\beta^*)$ is a representation of $(\mathcal{A},[~,~],N)$ if and only if, for all $x\in \mathcal{A}, v\in V$,
 \begin{eqnarray}\label{eq:fb}
 &\beta(\rho(N(x))v)+\rho(x)\beta^2(v)=\rho(N(x))\beta (v)+\beta(\rho(x) \beta(v)).&
 \end{eqnarray}
 In particular, for a linear map $S:\mathcal{A}\longrightarrow \mathcal{A}$, the triple $(\mathcal{A}^*, ad^*, S^*)$ is a representation of $(\mathcal{A},[~,~],N)$ if and only if, for all $x, y\in \mathcal{A}$,
 \begin{eqnarray}\label{eq:sfh}
 S([N(x),y])+[x,S^2(y)]=[N(x),S(y)]+S([x,S(y)]).
 \end{eqnarray}
 \end{lem}

 \begin{proof} We note that $(V^*,\rho^*)$ is a representation of $(\mathcal{A},[~,~])$. For all $x\in \mathcal{A}, u^*\in V^*, v\in V$, we check Eq.$(\ref{eq:1})$ for $(V^*,\rho^*,\beta^*)$ as follows.
 \begin{eqnarray*}
 &&\hspace{-16mm}\langle\rho^* (N(x))\beta^*(u^*)+{\beta^*}^2(\rho^*(x)u^*)-\beta^*(\rho^*(N(x)u^*)-\beta^*(\rho^*(x) \beta^*(u^*)), v\rangle\\
 &\stackrel {(\ref{eq:sr})}{=}&\langle u^*,\beta(\rho(N(x))v)+\rho(x)(\beta^2(v))-\rho(N(x))\beta(v)-\beta(\rho(x)\beta(v))\rangle.
 \end{eqnarray*}
 The rest are obvious.
 \end{proof}


 \begin{defi}\label{de:mrt} With notations in Lemma \ref{lem:HO}. If Eq.(\ref{eq:fb}) holds, then we say that $\beta$ is {\bf admissible} to $(\mathcal{A},[~,~],N)$ on $(V, \rho)$. If Eq.(\ref{eq:sfh}) holds, we say that $S$ is {\bf adjoint-admissible} to $(\mathcal{A},[~,~],N)$ or $(\mathcal{A},[~,~],N)$ is {\bf $S$-adjoint-admissible}.
 \end{defi}

\subsection{Mock-Lie bialgebras revisited}\mlabel{se:tri}

 Dual to the notion of mock-Lie algebras, we have

 \begin{defi}\mlabel{de:cl} A {\bf mock-Lie coalgebra} is a pair $(\mathcal{A}, \D)$, where $\mathcal{A}$ is a vector space and $\D: \mathcal{A}\lr \mathcal{A}\o \mathcal{A}$ (Sweedler notation \cite{Sw}: $\D(x)=x_{(1)}\o x_{(2)}$ ) is a linear map such that for all $x\in\mathcal{A}$,
 \begin{eqnarray}
 &x_{(1)}\o x_{(2)}=x_{(2)}\o x_{(1)},&\mlabel{eq:cml-1}\\
 &x_{(1)}\o x_{(2)(1)}\o x_{(2)(2)}+x_{(1)(1)}\o x_{(1)(2)}\o x_{(2)}+x_{(2)(1)}\o x_{(1)}\o x_{(2)(2)}=0.&\mlabel{eq:cml}
 \end{eqnarray}
 \end{defi}

 \begin{pro}\mlabel{pro:de:cl} \begin{enumerate}[(1)]
   \item \mlabel{it:de:cl1} If $(\mathcal{A}, \D)$ is a mock-Lie coalgebra, then $(\mathcal{A}^*, \D^*)$ is a mock-Lie algebra, where $\D^*:\mathcal{A}^*\o \mathcal{A}^*\lr \mathcal{A}^*$ is defined by
   $$
   \langle\D^*(\xi\o \eta), x\rangle=\langle \xi, x_{(1)}\rangle \langle \eta, x_{(2)}\rangle, \quad \forall~~\xi, \eta\in \mathcal{A}^*, x\in \mathcal{A}.
   $$
   \item \mlabel{it:de:cl2} If $(\mathcal{A}, [~,~])$ is a finite-dimensional mock-Lie algebra, then $(\mathcal{A}^*, [~,~]^*)$ is a mock-Lie coalgebra, where $[~,~]^*:\mathcal{A}^*\lr \mathcal{A}^*\o \mathcal{A}^*$ is defined by
   $$
   \langle[~,~]^*(\xi), x\o y\rangle=\langle \xi, [x, y]\rangle, \quad \forall~~\xi\in \mathcal{A}^*, x, y\in \mathcal{A}.
   $$
 \end{enumerate}
 \end{pro}

 \begin{proof} Similar to the case of Hopf algebra, see \cite{Sw}. \end{proof}

 \begin{defi}\mlabel{de:lb} A {\bf mock-Lie bialgebra} is a triple $(\mathcal{A}, [~,~], \D)$, where $(\mathcal{A}, [~,~])$ is a mock-Lie algebra and $(\mathcal{A}, \D)$ is a mock-Lie coalgebra such that, for all $x, y\in \mathcal{A}$,
 \begin{eqnarray}
 &\D([x, y])=-[x, y_{(1)}]\o y_{(2)}-y_{(1)}\o [x, y_{(2)}]-[y, x_{(1)}]\o x_{(2)}-x_{(1)}\o [y, x_{(2)}].& \mlabel{eq:b2}
 \end{eqnarray}
 \end{defi}

 \begin{rmk} \mlabel{rmk:de:lb} In Definition \mref{de:lb}, the condition that $(\mathcal{A}^*, [~,~]_{\mathcal{A}^*})$ is a mock-Lie algebra in \cite[Definition 3.5]{BCHM} is replaced by the condition that $(\mathcal{A}, \D)$ is a mock-Lie coalgebra. There is no dimensional limit, and Definition \mref{de:lb} is more common than \cite[Definition 3.5]{BCHM}.
 \end{rmk}

 Let $(\mathcal{A}, [~,~])$ be a mock-Lie algebra. For $r\in \mathcal{A}\o \mathcal{A}$, define $\D_r$ by
 \begin{equation}\mlabel{eq:cop}
 \D(x):=\D_r(x):=[x, r^{1}]\o r^{2}-r^{1}\o [x, r^{2}], \quad \forall x\in \mathcal{A}.
 \end{equation}

 \begin{pro} \label{de:qt} Let $(\mathcal{A}, [~,~])$ be a mock-Lie algebra. If $r\in \mathcal{A}\o \mathcal{A}$ is an antisymmetric solution of the following {\bf classical mock-Lie Yang-Baxter equation} (cmLYBe for short) in $(\mathcal{A}, [~,~])$:
 \begin{eqnarray}
 &[r^1, \br^1]\o r^2\o \br^2+r^1\o \br^1\o [r^2,\br^2]-r^1\o [\br^1, r^2]\o \br^2=0,&\mlabel{eq:cybe}
 \end{eqnarray}
 then $(\mathcal{A}, [~,~], \D_r)$ is a mock-Lie bialgebra, where $\D_r$ is defined by Eq.(\mref{eq:cop}). In this case, we call this mock-Lie bialgebra {\bf quasitriangular} and it is denoted by $(\mathcal{A}, [~,~], r, \D_r)$.
 \end{pro}

 \begin{proof} It is direct by \cite[Theorem 4.4]{BCHM}. \end{proof}

 \begin{pro}\mlabel{pro:qt} Let $(\mathcal{A}, [~,~])$ be a mock-Lie algebra and $r\in \mathcal{A}\o \mathcal{A}$ be an antisymmetric element. Then the quadruple $(\mathcal{A}, [~,~], r, \D_r)$ is a quasitriangular mock-Lie bialgebra, where $\D_r$ is defined by Eq.(\mref{eq:cop}), if and only if
 \begin{eqnarray}
 &(\D_r\o \id)(r)=-r^1\o \br^1\o [r^2, \br^2]&\mlabel{eq:qt1}
 \end{eqnarray}
 or
 \begin{eqnarray}
 &(\id\o \D_r)(r)=[r^1, \br^1]\o r^2\o \br^2&\mlabel{eq:qt2}
 \end{eqnarray}
 holds.
 \end{pro}

 \begin{proof} By Eq.(\mref{eq:cop}), we have
 \begin{eqnarray*}\mlabel{eq:qt1-1}
 (\D_r\o \id)(r)
 &\stackrel{ }=&[r^1, \br^1]\o \br^2\o r^2-r^1\o [\br^2, r^2]\o \br^2\\
 &\stackrel{(\ref{eq:1-1})}=&[r^1, \br^1]\o r^2\o \br^2-r^1\o [\br^2, r^2]\o \br^2.
 \end{eqnarray*}
 Then Eq.(\ref{eq:cybe}) $\Leftrightarrow$ Eq.(\ref{eq:qt1}).
 By the skew-symmetry of $r$, we obtain that Eq.(\mref{eq:qt2}) $\Leftrightarrow$  Eq.(\mref{eq:cybe}).
 \end{proof}

 \begin{defi}\mlabel{de:ccybe} Let $(\mathcal{A}, \D)$ be a mock-Lie coalgebra and $\om\in (\mathcal{A}\o \mathcal{A})^*$ be a bilinear form. The equation
 \begin{eqnarray}\mlabel{eq:ccybe}
 \om(x_{(1)}, y)\om(x_{(2)}, z)+\om(x, z_{(1)})\om(y, z_{(2)})-\om(x, y_{(1)})\om(y_{(2)}, z)=0,\quad x, y, z\in \mathcal{A},
 \end{eqnarray}
 is called a {\bf co-classical mock-Lie Yang-Baxter equation (ccmLYBe for short) in $(\mathcal{A}, \D)$}.
 \end{defi}

 \begin{thm}\mlabel{de:cqt} Let $(\mathcal{A}, \D)$ be a mock-Lie coalgebra and $\om\in (\mathcal{A}\o \mathcal{A})^*$ be a skew-symmetric (in the sense of $\om(x, y)=-\om(y, x)$) solution of the ccmLYBe in $(\mathcal{A}, \D)$.
 Then $(\mathcal{A}, [~,~]_{\om}, \D)$ is a mock-Lie bialgebra, where the multiplication on $\mathcal{A}$ is defined by
 \begin{eqnarray}\mlabel{eq:p}
 [x, y]_{\om}=x_{(1)}\om(x_{(2)}, y)-y_{(1)}\om(x, y_{(2)}),\quad \forall~x, y\in \mathcal{A}.
 \end{eqnarray}
 This bialgebra is called a {\bf dual quasitriangular mock-Lie bialgebra} denoted by $(\mathcal{A}, \D, \om, [~,~]_{\om})$.
 \end{thm}

 \begin{proof} We check that Jacobi identity holds for $[~,~]_\om$ as follows, others conditions are straightforward. For all $x, y, z\in \mathcal{A}$, we have
 \begin{eqnarray*}
 &&\hspace{-10mm}[x, [y, z]_{\om}]_{\om}+[y, [z, x]_{\om}]_{\om}+[z, [x, y]_{\om}]_{\om}\\
 &&\stackrel{(\mref{eq:p})}{=}x_{(1)}\om(x_{(2)}, y_{(1)})\om(y_{(2)}, z)-y_{(1)(1)}\om(x, y_{(1)(2)})\om(y_{(2)}, z)-x_{(1)}\om(x_{(2)}, z_{(1)})\om(y, z_{(2)})\\
 &&\hspace{6mm}+z_{(1)(1)}\om(x, z_{(1)(2)})\om(y, z_{(2)})+y_{(1)}\om(y_{(2)}, z_{(1)})\om(z_{(2)}, x)-z_{(1)(1)}\om(y, z_{(1)(2)})\om(z_{(2)}, x)\\
 &&\hspace{6mm}-y_{(1)}\om(y_{(2)}, x_{(1)})\om(z, x_{(2)})+x_{(1)(1)}\om(y, x_{(1)(2)})\om(z, x_{(2)})+z_{(1)}\om(z_{(2)}, x_{(1)})\om(x_{(2)}, y)\\
 &&\hspace{6mm}-x_{(1)(1)}\om(z, x_{(1)(2)})\om(x_{(2)}, y)-z_{(1)}\om(z_{(2)}, y_{(1)})\om(x, y_{(2)})+y_{(1)(1)}\om(z, y_{(1)(2)})\om(x, y_{(2)})\\
 &&\stackrel{(\mref{eq:ccybe})}{=}x_{(1)}\om(x_{(2)(1)}, y)\om(x_{(2)(2)}, z)-y_{(1)(1)}\om(x, y_{(1)(2)})\om(y_{(2)}, z)+z_{(1)(1)}\om(x, z_{(1)(2)})\om(y, z_{(2)})\\
 &&\hspace{6mm}+y_{(1)}\om(y_{(2)(1)}, z)\om(y_{(2)(2)}, x)-z_{(1)(1)}\om(y, z_{(1)(2)})\om(z_{(2)}, x)+x_{(1)(1)}\om(y, x_{(1)(2)})\om(z, x_{(2)})\\
 &&\hspace{6mm}+z_{(1)}\om(z_{(2)(1)}, x)\om(z_{(2)(2)}, y)-x_{(1)(1)}\om(z, x_{(1)(2)})\om(x_{(2)}, y)+y_{(1)(1)}\om(z, y_{(1)(2)})\om(x, y_{(2)})\\
 &&\stackrel{(\mref{eq:cml-1})(\mref{eq:ccybe})}{=}0,
 \end{eqnarray*}
 finishing the proof.
 \end{proof}

 \begin{pro}\mlabel{pro:cqt} Let $(\mathcal{A}, \D)$ be a mock-Lie coalgebra and $\om\in (\mathcal{A}\o \mathcal{A})^*$ be a skew-symmetric 2-form. Then the quadruple $(\mathcal{A}, \D, \om, [~,~]_{\om})$ is a dual quasitriangular mock-Lie bialgebra, where $[~,~]_{\om}$ is defined by Eq.(\mref{eq:p}), if and only if
 \begin{eqnarray}\mlabel{eq:cqt1}
 \om([x, y]_{\om}, z)=-\om(x, z_{(1)})\om(y, z_{(2)})
 \end{eqnarray}
 or
 \begin{eqnarray}\mlabel{eq:cqt2}
 \om(x, [y, z]_{\om})=\om(x_{(1)}, y)\om(x_{(2)}, z)
 \end{eqnarray}
 holds, where $x, y, z\in \mathcal{A}$.
 \end{pro}

 \begin{proof} It is direct by Eqs. \eqref{eq:p} and \eqref{eq:ccybe}.
 \end{proof}

\subsection{Symplectic mock-Lie algebras from dual quasitriangular mock-Lie bialgebras}

 \begin{defi}(\cite[Definition 5.1]{BCHM})\mlabel{de:sym} Let $(\mathcal{A}, [~,~])$ be a mock-Lie algebra and $\om\in (\mathcal{A}\o \mathcal{A})^*$ be a skew-symmetric 2-form. If for all $x, y, z\in \mathcal{A}$, the equation below holds:
 \begin{eqnarray}\mlabel{eq:symp}
 \om([x, y], z)+\om([y, z], x)+\om([z, x], y)=0,
 \end{eqnarray}
 then $(\mathcal{A}, [~,~])$ is said a {\bf symplectic mock-Lie algebra} denoted by $(\mathcal{A}, [~,~], \om)$.
 \end{defi}

 \begin{pro}\mlabel{pro:de:sym} Let $(\mathcal{A}, \D)$ be a mock-Lie coalgebra and $\om\in (\mathcal{A}\o \mathcal{A})^*$ be a skew-symmetric 2-form. If $(\mathcal{A}, \D, \om, [~,~]_{\om})$ is a dual quasitriangular mock-Lie bialgebra, where $[~,~]_{\om}$ is defined by Eq.(\mref{eq:p}), then $(\mathcal{A}, [~,~]_\om, \om)$ is a symplectic mock-Lie algebra.
 \end{pro}

 \begin{proof} For all $x, y, z\in \mathcal{A}$, we have
 \begin{eqnarray*}
 &&\hspace{-13mm}\om([x, y]_\om, z)+\om([y, z]_\om, x)+\om([z, x]_\om, y)\\
 &\stackrel{(\mref{eq:cqt1})}{=}&-\om(x, z_{(1)})\om(y, z_{(2)})-\om(y, x_{(1)})\om(z, x_{(2)})-\om(z, y_{(1)})\om(x, y_{(2)})\\
 &\stackrel{(\ref{eq:cml-1})}{=}&-\om(z_{(1)}, x)\om(z_{(2)}, y)-\om(x_{(1)}, y)\om(x_{(2)}, z)+\om(x, y_{(1)})\om(y_{(2)}, z)\\
 &&\hspace{10mm} \hbox{~(by the skew-symmetry of } \om)\\
 &\stackrel{(\mref{eq:ccybe})}{=}&0,
 \end{eqnarray*}
 as desired.
 \end{proof}

\subsection{\N operators from symplectic mock-Lie algebras}

 \begin{thm}\mlabel{thm:ln} Let $(\mathcal{A}, [~,~], \om)$ be a symplectic mock-Lie algebra and $r\in \mathcal{A}\o \mathcal{A}$ be an element. If $(\mathcal{A}, [~,~], r, \D_r)$ is a quasitriangular mock-Lie bialgebra together with the comultiplication $\D_r$ defined in Eq.(\mref{eq:cop}) and further $(\mathcal{A}, \D_r, \om, [~,~]_{\om})$ is a dual quasitriangular mock-Lie bialgebra together with the multiplication $[~,~]_{\om}$ given in Eq.(\mref{eq:p}).
 Then $(\mathcal{A}, [~,~], N)$ is a \N mock-Lie algebra, where $N:\mathcal{A}\lr \mathcal{A}$ is given by
 \begin{eqnarray}
 &N(x)=\om(x, r^1)r^2,~~\forall~x\in \mathcal{A}.&\mlabel{eq:ys}
 \end{eqnarray}
 \end{thm}

 \begin{proof} By Eqs.(\mref{eq:ccybe}) and (\mref{eq:cop}), for all $x, y, z\in \mathcal{A}$, one has
 \begin{eqnarray*}
 &\hspace{-15mm}0=\om([x,r^1], y)\om(r^2, z)-\om(r^1, y)\om([x,r^2], z)
 +\om(x, [z,r^1])\om(y, r^2)\\
 &-\om(x, r^1)\om(y, [z,r^2])-\om(x, [y,r^1])\om(r^2, z)+\om(x,r^1)\om([y,r^2], z)&\\
 &\hspace{-12mm}\stackrel{(\ref{eq:1-1})}{=}\om([x, r^1], y)\om(r^2, z)-\om(r^1, y)\om([x,r^2], z)-\om(r^1, y)\om([r^2,z], x)&\\
 &+\om(x, r^1)\om([r^2,z], y)+\om([r^1,y], x)\om(r^2, z)+\om(x, r^1)\om([y, r^2], z)&\\
 &\hspace{7mm} \hbox{~(by the skew-symmetry of } \om \hbox{~and the antisymmetry of } r)&
 \end{eqnarray*}
 Since, further, $(\mathcal{A}, [~,~], \om)$ is a symplectic mock-Lie algebra, we obtain
 \begin{eqnarray}
 &-\om([y,x], r^1)\om(r^2, z)+\om(r^1, y)\om([z,x], r^2)-\om(x, r^1)\om([z,y], r^2)=0.&\mlabel{eq:cqt-sym}
 \end{eqnarray}
 Now we can check that $N$ defined in Eq.(\mref{eq:ys}) is a Nijenhuis operator on $(\mathcal{A}, [~,~])$.
 \begin{eqnarray*}
 &&\hspace{-15mm}[N(x), N(y)]-N([N(x), y]+[x, N(y)]-N([x, y]))\\
 &\stackrel{(\mref{eq:ys})}{=}&\om(x, r^1)\om(y, \bar{r}^1)[r^2, \bar{r}^2]
 -\om(x, r^1)\om([r^2, y], \bar{r}^1)\bar{r}^2-\om(y, r^1)\om([x, r^2], \bar{r}^1)\bar{r}^2\\
 &&+\om([x, y], r^1)\om(r^2, \bar{r}^1)\bar{r}^2\\
 &\stackrel{(\mref{eq:cqt-sym})}{=}&\om(x, r^1)\om(y, \bar{r}^1)[r^2, \bar{r}^2]
 -\om(x, r^1)\om([r^2, y], \bar{r}^1)\bar{r}^2-\om(y, r^1)\om([x, r^2], \bar{r}^1)\bar{r}^2\\
 &&+\om(r^1, y)\om([\br^1, x], r^2)\br^2-\om(x, r^1)\om([\br^1, y], r^2)\br^2\\
 &\stackrel{(\mref{eq:symp})}{=}&\om(x, r^1)\om(y, \bar{r}^1)[r^2, \bar{r}^2]+\om(x, r^1)\om([\br^1, r^2], y)\br^2+\om(y, r^1)\om([r^2, \br^1], x)\br^2\\
 &\stackrel{(\mref{eq:cybe})}{=}&0. \hbox{~(also by the skew-symmetry of } \om \hbox{~and the antisymmetry of } r).
 \end{eqnarray*}
 These finish the proof.
 \end{proof}

 \begin{ex}\mlabel{ex:thm:ln} Let $(\mathcal{A},[~,~])$ be a 4-dimensional mock-Lie algebra with basis $\{e_1,e_2,e_3,e_4\}$ and the product $[~,~]$ given by the following table
 \begin{center}
        \begin{tabular}{r|rrrr}
          $[~,~]$ & $e_1$  & $e_2$ & $e_3$ & $e_4$  \\
          \hline
           $e_1$ & $e_2$  & $0$  & $e_4$ & $0$\\
           $e_2$ & $0$  &  $0$  & $0$ & $0$\\
           $e_3$ & $e_4$ & $0$ & $0$ & $0$\\
           $e_4$ & $0$ & $0$ & $0$ & $0$\\
        \end{tabular}.
        \end{center}
       Set $r=e_2 \o e_3-e_3 \o e_2$, then $(\mathcal{A},[~,~],r,\D_r)$ is a quasitriangular mock-Lie bialgebra, where $\D_r$ is given by:
        $$\left\{
            \begin{array}{l}
             \D_r(e_1)=-e_4\o e_2-e_2\o e_4\\
             \D_r(e_2)=0\\
             \D_r(e_3)=0\\
             \D_r(e_4)=0\\
            \end{array}
            \right..$$
 Set
 \begin{center}
        \begin{tabular}{r|rrrr}
          $\omega$ & $e_1$  & $e_2$ & $e_3$ & $e_4$  \\
          \hline
           $e_1$ & $0$  & $0$  & $\l$ & $\g$\\
           $e_2$ & $0$  &  $0$  & $2\g$ & $0$\\
           $e_3$ & $-\l$ & $2\g$ & $0$ & $0$\\
           $e_4$ & $-\g$ & $0$ & $0$ & $0$\\
        \end{tabular}\quad ($\l, \g$ are parameters).
        \end{center}
 Then $(\mathcal{A},[~,~],\omega)$ is a symplectic mock-Lie algebra. Specially, when
 \begin{center}
        \begin{tabular}{r|rrrr}
          $\omega$ & $e_1$  & $e_2$ & $e_3$ & $e_4$  \\
          \hline
           $e_1$ & $0$  & $0$  & $\l$ & $0$\\
           $e_2$ & $0$  &  $0$  & $0$ & $0$\\
           $e_3$ & $-\l$ & $0$ & $0$ & $0$\\
           $e_4$ & $0$ & $0$ & $0$ & $0$\\
        \end{tabular},
        \end{center}
 $(\mathcal{A},\D_r,\omega,[~,~]_\omega)$ is a dual quasitriangular mock-Lie bialgebra. Then by Theorem \ref{thm:ln}, let
   $$\left\{
            \begin{array}{l}
             N(e_1)=-\l e_2\\
             N(e_2)=0\\
             N(e_3)=0\\
             N(e_4)=0\\
            \end{array}
            \right., $$
 then $(\mathcal{A}, [~,~], N)$ is a Nijenhuis mock-Lie algebra.
 \end{ex}

 \begin{defi}\mlabel{de:csym} Let $(\mathcal{A}, \D)$ be a mock-Lie coalgebra and $r\in \mathcal{A}\o \mathcal{A}$ be an antisymmetric element. Assume that
 \begin{eqnarray}\mlabel{eq:csymp1}
 {r^1}_{(1)}\o {r^1}_{(2)}\o r^2+r^2\o {r^1}_{(1)}\o {r^1}_{(2)}+{r^1}_{(2)}\o r^2\o   {r^1}_{(1)}=0.
 \end{eqnarray}
 Then $(\mathcal{A}, \D)$ is called {\bf a cosymplectic mock-Lie coalgebra} denoted by $(\mathcal{A}, \D, r)$.
 \end{defi}

 \begin{pro}\mlabel{pro:de:csym} Let $(\mathcal{A}, [~,~])$ be a mock-Lie algebra and $r\in \mathcal{A}\o \mathcal{A}$ antisymmetric. If $(\mathcal{A}, [~,~], r, \D_r)$ is a quasitriangular mock-Lie bialgebra, where $\D_r$ is defined by Eq.(\mref{eq:cop}), then $(\mathcal{A}, \D_r, r)$ is a cosymplectic mock-Lie coalgebra.
 \end{pro}

 \begin{proof} Straightforward by Eqs.(\ref{eq:cybe}) and (\ref{eq:cop}).
 \end{proof}

 \begin{defi} A {\bf \N mock-Lie coalgebra} is a triple $(\mathcal{A}, \D, S)$, where $(\mathcal{A}, \D)$ is a mock-Lie coalgebra and $S$ is a \N operator on $(\mathcal{A}, \D)$, i.e., the following identity  holds:
 \begin{eqnarray}\mlabel{eq:gd}
 S(x_{(1)})\o S(x_{(2)})+S^2(x)_{(1)}\o S^2(x)_{(2)}=S(S(x)_{(1)})\o S(x)_{(2)}+S(x)_{(1)}\o S(S(x)_{(2)}),~~\forall~x\in \mathcal{A}.
 \end{eqnarray}
 \end{defi}

 \begin{thm}\mlabel{thm:cln} Let $(\mathcal{A}, \D, r)$ be a cosymplectic mock-Lie coalgebra and $\om\in (\mathcal{A}\o \mathcal{A})^*$. If $(\mathcal{A}, \D, \om$, $[~,~]_\om)$ is a dual quasitriangular mock-Lie bialgebra together with the multiplication $[~,~]_\om$ defined in Eq.(\mref{eq:p}) and further $(\mathcal{A}, [~,~]_\om, r, \D_r)$ is a quasitriangular mock-Lie bialgebra together with the comultiplication $\D_r$ defined in Eq.(\mref{eq:cop}).
 Then $(\mathcal{A}, \D, S)$ is a \N mock-Lie coalgebra, where $S:\mathcal{A}\lr \mathcal{A}$ is given by
 \begin{eqnarray}
 &S(x)=r^1\om(r^2, x),~~\forall~x\in \mathcal{A}.&\mlabel{eq:cys}
 \end{eqnarray}
 \end{thm}

 \begin{proof} Similar to Theorem \ref{thm:ln}.
 \end{proof}

 \section{Deformations and extensions of \N mock-Lie algebras}\label{Sec3} In this section, we explore deformations and extensions of Nijenhuis mock-Lie bialgebras. We  introduce first the concept of a one-parameter formal deformation for a mock-Lie algebra and the Nijenhuis operator, then consider infinitesimal deformations within the cohomology framework of Nijenhuis mock-Lie algebras. Subsequently, we investigate abelian extensions of Nijenhuis mock-Lie algebras equipped with a given representation. We refer to \cite{BB2} for the cohomology and deformations of mock-Lie algebras. The cohomology  theory for  Nijenhuis mock-Lie algebras is very similar, one considers, as 2-cochains, pairs consisting of 2-cochain of a mock-Lie algebra and a linear map. The 2-cocycles and cohomology classes are described below.
 
\subsection{Deformations of Nijenhuis mock-Lie algebras}\label{sec4}
In this section, we study one-parameter formal deformations of Nijenhuis mock-Lie algebra. We denote the bracket $[\cdot,\cdot]$ by $\mu.$

\begin{defi}
A one-parameter formal deformation of a Nijenhuis mock-Lie algebra $(\mathcal{A}_T,\mu)$ is a pair of two power series  $(\mu_t,N_t)$, where 
\[ \mu_t=\sum _{i=0}^{\infty}\mu_it^i, ~ \mu_{i} \in Hom(\mathcal{A}\otimes\mathcal{A},\mathcal{A}),~~~~ N_t= \sum_{i=0}^{\infty} N_it^i,~ N_i \in Hom(\mathcal{A},\mathcal{A}),\]
 such that $(\mathcal{A}[[t]],\mu_t,N_t)$ is a Nijenhuis mock-Lie algebra with $(\mu_0,N_0)=(\mu ,N)$, where $\mathcal{A}[[t]]$, the space of formal power series in $t$ with coefficients in $\mathcal{A}$, which is a $K[[t]]$ module, $K$ being the ground field of $(\mathcal{A},\mu,N)$.
\end{defi}
The above definition holds if and only if, for all $x,y,z \in \mathcal{A}$, the following conditions are satisfied 
\[\mu_t(x,\mu_t(y,z))+\mu_t(z,\mu_t(x,y))+\mu_t(y,\mu_t(z,x))=0,\]
and \[ \mu_t(N_t(x),N_t(y))=N_t(\mu_t(x,N_t(y))+\mu_t(N_t(x),y)-N_t(x,y)).\]

Expanding the above equations and equating the coefficients of $t^n$ from both sides, we have for $n\geq 0$
\begin{align}\label{JacobiDeforCondition}
\sum _{\substack{i+j=n \\i,j\geq 0}} \Big(\mu_i(x,\mu_j (y,z))+\mu_i(z,\mu _j(x,y))+ \mu_i(y, \mu_j (z,x))\Big)=0,
\end{align}
and 
\begin{align}\label{NijenDeforCondition}
\sum_{\substack{i+j+k=n \\ i,j,k \geq 0}} \mu_i(N_j(u),N_k(v))=\sum_{\substack{i+j+k=n \\ i,j,k \geq 0}}\Big(N_i(\mu_j(N_k(u),v))+N_i(\mu_j(u,N_k(v)))-N_iN_k(\mu_j(u,v))\Big).
\end{align}
Observe that for $n=0$, the above conditions are exactly the conditions in the definitions of  mock-Lie  algebra and Nijenhuis operator.

Setting $n=1$ in the equation \eqref{JacobiDeforCondition}, we get 
\begin{equation}
\mu (x,\mu_1(y,z))+\mu_1(x,\mu (y,z))+ \mu (z,\mu_1(x,y))+\mu_1(z,\mu (x,y)) \\
+\mu_1 (y, \mu (z,x))+\mu (y, \mu_1(z,x))=0.\label{CocycleCond1}   
\end{equation} 

Again, setting $n=1$ in \eqref{NijenDeforCondition}, we get 
\begin{align}\nonumber
&\mu_1(N(x_1),N(x_2))+\mu (N_1(x_1),N(x_2))+\mu (N(x_1),N_1(x_2)) \\\nonumber
-&N_1(\mu (N(x_1),x_2)) -N(\mu (N_1(x_1),x_2)) -N(\mu _1 (N(x_1),x_2)) \\ \nonumber
-&N_1(\mu (x_1,N(x_2)))-N(\mu _1 (x_1,N(x_2)))-N(\mu (x_1,T_1(x_2))) \\+&N_1N(\mu(x,y))+NN_1(\mu(x,y))+N^2(\mu_1(x,y))\
= \ 0.\label{CocycleCond2}
\end{align}
Thus, by the above conditions, the pair $(\mu_1,N_1)$ is called  a $2$-cocycle.
\begin{defi}
The infinitesimal of the deformation $(\mu_t, N_t)$ is the pair $(\mu_1, N_1)$. Suppose more  generally that $(\mu_n, N_n)$ is the first non-zero term of $(\mu_t, N_t)$ after $(\mu_0, N_0)$, such $(\mu_n, N_n)$ is called a $n$-infinitesimal of the deformation.
\end{defi}

\begin{thm}
Let $(\mu_t,N_t)$ be a  one-parameter formal deformation of a Nijenhuis mock-Lie algebra $(\mathcal{A},\mu,N)$. Then $n$-infinitesimal of the deformation is a $2$-cocycle.
\end{thm}

\begin{defi}
Let $(\mu_t,N_t)$ and $(\mu_t^{'},N_t^{'})$ be two one-parameter formal deformations of a Nijenhuis mock-Lie algebra $(\mathcal{A},\mu,N )$. A formal isomorphism from 
$(\mu_t,N_t)$ to $(\mu_t^{'},N_t^{'})$ is a formal power series $\psi _t=\sum _{i=0}\psi_i t^i : \mathcal{A}[[t]] \rightarrow \mathcal{A}[[t]]$, where $\psi_i: \mathcal{A} \rightarrow \mathcal{A}$ are linear maps with $\psi_0$ is the identity map on $\mathcal{A}$ and also the following conditions are satisfied.
\begin{align}\label{morph1}
&\psi_t \circ \mu^{'}_t=\mu_t \circ (\psi_t \otimes \psi_t),\\\label{morph2}
& \psi_t \circ N_t^{'}=N_{t} \circ \psi_t.  
\end{align}
In this case, we say that $(\mu_t,N_t)$ and $(\mu_t^{'},N_t^{'})$  are equivalent.
Note that Equations \eqref{morph1} and \eqref{morph2} can be written as follows respectively:
\begin{align}\label{MorphDefCond}
&\sum_{\substack {i+j=n \\ i,j\geq 0}}\psi _i(\mu_j^{'}(x,y))=\sum_{\substack {i+j+k=n \\ i,j,k\geq 0}}\mu_i(\psi_j(x),\psi_k(y)),~~ x,y \in \mathcal{A},\\
\label{NijMorphCond}
& \sum_{\substack {i+j=n \\ i,j\geq 0}}\psi _i \circ N^{'}_j=\sum_{\substack {i+j=n \\ i,j\geq 0}}  N_i \circ \psi _j.
\end{align}
\end{defi}
Let $\psi_t : (\mu_t,N_t) \rightarrow  (\mu_t^{'},N_t^{'})$ be a formal isomorphism. Now setting $n=1$ in Equations \eqref{MorphDefCond} and \eqref{NijMorphCond}, we get 
\begin{align*}
&\mu^{'}_1(x,y)=\mu_1(x,y)+\mu (x,\psi_1 (y))+\mu (\psi_1(x),y)-\psi_1(\mu (x,y)) ,~~\forall x,y \in \mathcal{A},\\
& N_1^{'}=N_1+N \circ \psi_1-\psi _1 \circ N.
\end{align*}
Therefore, 
the infinitesimals of two equivalent one-parameter formal deformations of Nijenhuis mock-Lie algebra $(\mathcal{A}_N , \mu )$ are in the same cohomology class. 

\subsection{Abelian extensions of Nijenhuis mock-Lie algebras}\label{sec5}
Let $(\mathcal{A}, [~,~], N)$ be a Nijenhuis mock-Lie algebra and $V$ be a vector space. Observe that if $N_V$ is a linear operator on the vector space $V$ and  if we define the bracket by $\mu(x,y)=0$ for all $x,y \in V$. Then $(V,\mu,N_V)$ has a structure of Nijenhuis mock-Lie algebra.
\begin{defi}
 An abelian extension of the Nijenhuis mock-Lie algebra $(\mathcal{A},[~,~],N)$ is a short exact sequence of morphisms of Nijenhuis mock-Lie algebra 
 \[
\begin{tikzcd}
0 \arrow[r] & (V,\mu,N_V) \arrow[r ,"i"] & (\hat{\mathcal{A}},[~,~]_{\wedge},\hat{N}) \arrow[r,"p"] & (\mathcal{A},[~,~],N) \arrow [r] & 0 
\end{tikzcd} ,
\]
that is, there exists a commutative diagram 
\[
\begin{tikzcd}
0 \arrow[r] & V \arrow[r ,"i"] \arrow[d,"N_V"]& \hat{\mathcal{A}} \arrow[d,"\hat{N}"]\arrow[r,"p"] & \mathcal{A} \arrow [r]\arrow[d,"N"]  & 0 \\
 0 \arrow[r] & V \arrow[r ,"i"] & \hat{\mathcal{A}} \arrow [r,"p"]  &  \mathcal{A} \arrow [r]  & 0
\end{tikzcd}
\]
where $\mu (a,b)=0$ for all $a,b \in V.$ In this case we say that $(\hat{\mathcal{A}},[~,~]_{\wedge},\hat{N})$ is an abelian extension of the Nijenhuis mock-Lie algebra $(\mathcal{A},[~,~],N)$ by $(V,\mu,N_V).$
 \end{defi}

\begin{defi}
Let $(\hat{\mathcal{A}},[~,~]_{\wedge_1},\hat{N_1})$ and $(\hat{\mathcal{A}}  ,[~,~]_{\wedge_2},\hat{N_2})$ be two abelian extensions of $(\mathcal{A},[~,~],N)$ by $(V,\mu,N_V)$. Then this two extensions are said to be isomorphic if there exists an isomorphism of Nijenhuis mock-Lie algebra $\xi : (\hat{\mathcal{A}},[~,~]_{\wedge_1},\hat{N_1}) \rightarrow  (\hat{\mathcal{A}}  ,[~,~]_{\wedge_2},\hat{N_2}) $ such that the following diagram is commutative: 
\[
\begin{tikzcd}
0 \arrow[r] & (V,\mu,N_V) \arrow[r ,"i"] \arrow[d,equal]& (\hat{\mathcal{A}}  ,[~,~]_{\wedge_1},\hat{N_1})\arrow[d,"\xi"]\arrow[r,"p"] & (\mathcal{A},[~,~],N) \arrow [r]\arrow[d,equal]  & 0 \\
 0 \arrow[r] & (V,\mu,N_V) \arrow[r ,"i"] & (\hat{\mathcal{A}}  ,[~,~]_{\wedge_2},\hat{N_2}) \arrow [r,"p"]  &  (\mathcal{A},[~,~],N) \arrow [r]  & 0.
\end{tikzcd}
\]
\end{defi}

\begin{defi}
A section of an abelian extension $(\hat{\mathcal{A}}  ,[~,~]_{\wedge},\hat{N})$ of $(\mathcal{A},[~,~],N)$ by $(V, \mu,N_V)$ is a linear map $s : \mathcal{A} \rightarrow \hat{\mathcal{A}}$ such that $p \circ s= id_{\mathcal{A}}.$
 \end{defi}
Let $(\hat{\mathcal{A}}  ,[~,~]_{\wedge},\hat{N})$ be an abelian extension of $(\mathcal{A},[~,~],N)$ by $(V,\mu,N_V)$ with a section $s: \mathcal{A} \rightarrow \hat{\mathcal{A}}$. Define $\bar{\rho}_V: \mathcal{A}  \rightarrow End(V)$  by  $$\bar{\rho}_V(x)(u)=[s(x),u]_{\wedge},\quad \forall x\in \mathcal{A}, u \in V.$$

\begin{thm}
Under the above notations, $(V,\bar{\rho}_V,N_V)$ is a representation of the  Nijenhuis mock-Lie algebra  $(\mathcal{A},[~,~],N).$ Moreover, this representation is independent of the choice of sections.
\end{thm}
\begin{proof}
Since $s([x,y])-[s(x),s(y)]_{\wedge} \in V$, for all $x,y \in \mathcal{A},$. Therefore, let $x,y \in \mathcal{A}, u \in V$,  we have
\begin{align*}
&\bar{\rho}_V(x)\bar{\rho}_V(y)(u))+\bar{\rho}_V([x,y])(u)+\bar{\rho}_V(y)\bar{\rho}_V(x)(u) \\
&= \bar{\rho}_V(x)([s(y),u]_{\wedge})+[s([x,y]),u]_{\wedge}+\bar{\rho}_V(y)([s(x),u]_{\wedge})\\
&=[s(x),[s(y),u]_{\wedge}]_{\wedge}+[s([x,y]),u]_{\wedge}+[s(y),[s(x),u]_{\wedge}]_{\wedge}\\
&=[s(x),[s(y),u]_{\wedge}]_{\wedge}+[[s(x),s(y)]_{\wedge},u]_{\wedge}+[s(y),[s(x),u]_{\wedge}]_{\wedge}\\
&=0.
\end{align*}
 Hence, $(V,\bar{\rho}_V)$ is a representation of the mock-Lie algebra $(\mathcal{A}, [~,~]).$
Now, $s(N(x))-\hat{N}(s(x)) \in V$  for all $x\in \mathcal{A}$. Therefore, we have
\begin{align*}
\bar{\rho}_V(N(x))N_V(u)&=[s(N(x)), N_V(u)]_{\wedge}=[\hat{N}(s(x)),\hat{N}(u)]_{\wedge}\\
&=\hat{N}\bigg([\hat{N}(s(x)),u]_{\wedge}+[s(x),\hat{N}(u)]_{\wedge}-\hat{N}([s(x),u]_{\wedge})\bigg)\\
&=N_V\bigg([s(N(x)),u]_{\wedge}+[s(x),N_V(u)]_{\wedge}-N_v(\bar{\rho}_V(x)(u))\bigg)\\
&=N_V\bigg(\bar{\rho}_V(N(x))(u)+\bar{\rho}_V(x)(N_V(u))-N_v(\bar{\rho}_V(x)(u))\bigg).
\end{align*}

Hence, $\bar{\rho}_V(N(x))N_V(u)=N_V\bigg(\bar{\rho}_V(N(x))(u)+\bar{\rho}_V(x)(N_V(u))-N_v(\bar{\rho}_V(x)(u))\bigg)$ for all $x,y \in \mathcal{A}$ and $u\in V.$ Thus $(V,\bar{\rho}_V,N_V)$ is a representation of the  Nijenhuis mock-Lie algebra  $(\mathcal{A},[~,~],N).$
Let  $s_2$ be  another
section, then $s_1(x)-s_2(x)\in V,~\mbox{for all }~ x \in \mathcal{A}.$
Thus, 
$
[s_1(x)-s_2(x),u]_{\wedge}=0$.
Then, two distinct sections give the same representation.
\end{proof}
Now, define two linear maps $\psi : \mathcal{A}\otimes \mathcal{A} \rightarrow V$ and $\chi : \mathcal{A} \rightarrow V$ by  
$
\psi (x \otimes y)=[s(x),s(y)]_{\wedge}-s([x,y])
\quad \textup{and} \quad
\chi (x)=\hat{N}(s(x))-s(N(x))
$ respectively. Let $s_1$ and $s_2$ be two distinct sections. Define $\gamma : \mathcal{A} \rightarrow V$ by
$\gamma(x)=s_1(x)-s_2(x),~\mbox{for all }~ x \in \mathcal{A}$.
Since $\mu (u,v)=0$ for all $ u,v \in V,$
therefore, for all $x,y \in \mathcal{A}, u \in V,$ we have
\begin{align*}
\psi_1(x,y)
&=[s_1(x),s_1(y)]_{\wedge}-s_1([x,y])\\
&=[s_2(x)+\gamma(x),s_2(y)+\gamma(y)]_{\wedge}-(s_2([x,y])+\gamma ([x,y]))\\
&=[s_2(x),s_2(y)]_{\wedge}-s_2([x,y])+[\gamma(x),s_2(y)]_{\wedge}+[s_2(x),\gamma (y)]_{\wedge}-\gamma([x,y])\\
&=\psi_2(x,y)+\delta^1 (\gamma)(x,y).
\end{align*}
Also,
\begin{align*}
\chi_1(x)
&=\hat{N}(s_1(x))-s_1(N(x))\\
&=\hat{N}(s_2(x)+\gamma(x))-s_2(N(x))-\gamma(N(x))\\
&=\chi_2(x)+N_V(\gamma(x))-\gamma (N(x))\\
&=\chi_2(x)-\phi^1(\gamma)(x).
\end{align*}
Therefore, $(\psi_1,\chi_1)-(\psi_2,\chi_2)=(\delta^1 (\gamma),~-\phi^1(\gamma))=d^1(\gamma)$. We say that $(\psi_1,\chi_1)$ and $(\psi_2,\chi_2)$ are in the same    cohomology class of Nijenhuis mock-Lie algebra $(\mathcal{A} ,[~,~], N)$ with  coefficients in $V$. 
 Then we have the following result.
\begin{pro}
The cohomology class of $(\psi,\chi)$ does not depend on the choice of sections.
\end{pro}
\begin{thm}
Let $V$ be a vector space and $N_V:V \rightarrow V$ be a linear map. Then $(V,\mu,N_V)$ is a Nijenhuis mock-Lie algebra with the bracket $\mu (u,v)=0 $ for all $u,v \in V.$ Then two isomorphic abelian extensions of a Nijenhuis mock-Lie algebra $(\mathcal{A} ,[~,~], N)$ by $(V,\mu,N_V)$ give rise to the same cohomology class.
\end{thm}
\begin{proof}
Let $(\hat{\mathcal{A}} ,[~,~]_{\wedge_1},\hat{N_1})$ and $(\hat{\mathcal{A}} ,[~,~]_{\wedge_2},\hat{N_2})$ be two isomorphic abelian extensions of $(\mathcal{A},[~,~])$ by $(V_N,\mu)$. Let $s_1$ be a section of $(\hat{\mathcal{A}}_{\hat{N_1}},[~,~]_{\wedge_1})$. Thus, we have 
$p_2 \circ (\xi \circ s_1)=p_1 \circ s_1 =id_{\mathcal{A}}$ as $p_2 \circ \xi =p_1$, where $\xi$ is the map between the two abelian extensions. Hence $\xi \circ s_1$ is a section of $(\hat{\mathcal{A}}_{\hat{N_2}},[~,~]_{\wedge_2})$.
Now define $s_2 :=\xi \circ s_1$. Since $\xi$ is a homomorphism of Nijenhuis mock-Lie algebras such that 
$\xi|_{V}=id_V, ~\xi ([s_1(x),u]_{\wedge_1})=[s_2(x),u]_{\wedge_2}$ . Thus, $\xi|_{V}: V \rightarrow V$ is compatible with the induced representations.

Now, for all $x,y \in \mathcal{A}$, we have 
\begin{align*}
\psi_2(x \otimes y)
&=[s_2(x),s_2(y)]_{\wedge_2}-s_2([x,y])\\
&=[\xi(s_1(x)),\xi(s_1(y)])_{\wedge_2}-\xi(s_1([x,y]))\\
&=\xi ([s_1(x),s_1(y)]_{\wedge_1}-s_1([x,y]))\\
&=\xi (\psi_1 (x \otimes y))=\psi_1(x\otimes y),
\end{align*}
and
\begin{align*}
\chi_2(x)
&= \hat{N_2}(s_2(x))-s_2(N(x))\\
&=\hat{N_2}(\xi(s_1(x)))-\xi(s_1(N(x)))\\
&=\xi (\hat{N_1 }(s_1(x))-s_1(N(x)))\\
&=\xi(\chi_1(x))=\chi_1(x).
\end{align*}
Therefore, two isomorphic abelian extensions give rise to the same cohomology class.
\end{proof}


 \section{\N mock-Lie bialgebras}\label{Sec4}
In this section, we introduce and study matched pairs and Manin triples of Nijenhuis mock-Lie algebras, establishing their connection with Nijenhuis mock-Lie bialgebras. Additionally, we explore the class of coboundary Nijenhuis mock-Lie bialgebras through the framework of 
$s$-adjoint-admissible Nijenhuis mock-Lie algebras, the $S$-admissible mock-Lie Yang-Baxter equation and
$\mathcal O$-operators.

\subsection{Matched pairs of Nijenhuis mock-Lie algebras}
 \begin{pro}(\cite[Theorem 3.1]{BCHM})\label{pro:aa} Let $(\mathcal{A},[~,~]_\mathcal{A})$ and $(\mathcal{A'},[~,~]_\mathcal{A'})$ be two mock-Lie algebras, $\rho_\mathcal{A}:\mathcal{A}\longrightarrow End(\mathcal{A'})$, $\rho_\mathcal{A'}:\mathcal{A'}\longrightarrow End(\mathcal{A})$ be two linear maps. Define the bracket $[~,~]_\diamond$ on the direct sum $\mathcal{A} \oplus \mathcal{A'}$ by
 \begin{eqnarray}\label{eq:sff}
 [(a+x),(b+y)]_\diamond:=[a,b]_\mathcal{A} + {\rho_\mathcal{A'}}(y)a +{\rho_\mathcal{A'}}(x)b + [x,y]_\mathcal{A'} +{\rho_\mathcal{A}}(a)y +{\rho_\mathcal{A}}(b)x,
 \end{eqnarray}
 where $a,b\in \mathcal{A} $ and $x,y\in \mathcal{A'}$. Then $(\mathcal{A}\oplus \mathcal{A'} , [~,~]_\diamond)$ is a mock-Lie algebra if and only if $(\mathcal{A}, \mathcal{A'},\rho_\mathcal{A},\rho_\mathcal{A'} )$ is a matched pair of mock-Lie algebras, that is,  $(\mathcal{A'},\rho_\mathcal{A'})$ is a representation of $(\mathcal{A},[~,~]_\mathcal{A})$, $(\mathcal{A},\rho_\mathcal{A})$ is a representation of $(\mathcal{A'},[~,~]_\mathcal{A'})$ and the following two equations hold:
 \begin{eqnarray*}
 &{\rho_\mathcal{A}}(a)[x,y]_\mathcal{A'} + [{\rho_\mathcal{A}}(a)x,y]_\mathcal{A'} + [x,{\rho_\mathcal{A}}(a)y]_\mathcal{A'} +{\rho_\mathcal{A}}({\rho_\mathcal{A}}(y)a)x +{\rho_\mathcal{A}}({\rho_\mathcal{A}}(x)a)y=0,&\label{eq:spf}\\
 &{\rho_\mathcal{A'}}(x)[a,b]_\mathcal{A} + [{\rho_\mathcal{A'}}(x)a,b]_\mathcal{A} + [a,{\rho_\mathcal{A'}}(x)b]_\mathcal{A} +{\rho_\mathcal{A'}}({\rho_\mathcal{A'}}(b)x)a +{\rho_\mathcal{A'}}({\rho_\mathcal{A'}}(a)x)b=0.&\label{eq:skf}
 \end{eqnarray*}
 \end{pro}

 Next we give the notion of a matched pair of Nijenhuis mock-Lie algebras which extends the one of mock-Lie algebras.
 \begin{defi}\label{de:mv} A {\bf matched pair of Nijenhuis mock-Lie algebras} $(\mathcal{A},[~,~]_\mathcal{A},N_\mathcal{A})$ and  $(\mathcal{A'},[~,~]_\mathcal{A'},N_\mathcal{A'})$ is a six-tuple $((\mathcal{A}, N_\mathcal{A}), (\mathcal{A'}, N_\mathcal{A'}),\rho_\mathcal{A}, \rho_\mathcal{A'})$ where $(\mathcal{A},\rho_\mathcal{A} ,N_\mathcal{A})$ is a representation of $(\mathcal{A'},[~,~]_\mathcal{A'},N_\mathcal{A'})$ and $(\mathcal{A'},\rho_\mathcal{A'} ,N_\mathcal{A'})$ is a representation of $(\mathcal{A},[~,~]_\mathcal{A},N_\mathcal{A})$, and $( \mathcal{A},\mathcal{A'},\rho_\mathcal{A}, \rho_\mathcal{A'} )$ is a matched pair of mock-Lie algebras.
 \end{defi}

 \begin{thm}\label{thm:gp} Let $(\mathcal{A},[~,~]_\mathcal{A},N_\mathcal{A})$ and  $(\mathcal{A'},[~,~]_\mathcal{A'},N_\mathcal{A'})$ be two Nijenhuis mock-Lie algebras. Then $((\mathcal{A}, N_\mathcal{A}), (\mathcal{A'}, N_\mathcal{A'}),\rho_\mathcal{A}, \rho_\mathcal{A'})$ is a matched pair of $(\mathcal{A},[~,~]_\mathcal{A},N_\mathcal{A})$ and $(\mathcal{A'},[~,~]_\mathcal{A'},N_\mathcal{A'})$ if and only if $(\mathcal{A}\oplus \mathcal{A'}, [~,~]_\diamond,N_\mathcal{A}+N_\mathcal{A'})$ is a Nijenhuis mock-Lie algebra by defining the multiplication on $\mathcal{A} \oplus \mathcal{A}$ by Eq.\eqref{eq:sff} and linear map $N_\mathcal{A}+N_\mathcal{A'} : \mathcal{A} \oplus\mathcal{A'} \longrightarrow \mathcal{A}\oplus\mathcal{A'}$ by
 \begin{eqnarray}\label{eq:cgy}
 (N_\mathcal{A}+N_\mathcal{A'})(a+x):=N_\mathcal{A}(a)+N_\mathcal{A'}(x),
 \end{eqnarray}
 for all $a,b\in\mathcal{A}$ and $x,y\in\mathcal{A'}$.
 \end{thm}

 \begin{proof} It can be proved by Proposition \mref{pro:aa} and the following equality 
 \begin{eqnarray*}
 &&\hspace{-15mm}[(N_\mathcal{A}+N_\mathcal{A'})(a+x),(N_\mathcal{A}+N_\mathcal{A'})(b+y)]_\diamond+(N_\mathcal{A}+N_\mathcal{A'})^2([(a+x),(b+y)]_\diamond)\\
 &&\hspace{-6mm}-(N_\mathcal{A}+N_\mathcal{A'})([(N_\mathcal{A}+N_\mathcal{A'})(a+x),(b+y)]_\diamond)-(N_\mathcal{A}
 +N_\mathcal{A'})([(a+x),(N_\mathcal{A}+N_\mathcal{A'})(b+y)]_\diamond)\\
 &\stackrel {(\ref{eq:cgy})(\ref{eq:sff})(\ref{eq:bja})}{=}&\hspace{-3mm}
 {\rho_\mathcal{A'}}(N_\mathcal{A'}(y))N_\mathcal{A}(a)
 +{\rho_\mathcal{A'}}(N_\mathcal{A'}(x))N_\mathcal{A}(b)+{\rho_\mathcal{A}}(N_\mathcal{A}(a))N_\mathcal{A'}(y)+{\rho_\mathcal{A}}(N_\mathcal{A}(b))N_\mathcal{A'}(x)\\
 &&\hspace{-5mm}+{N_\mathcal{A}}^2({\rho_\mathcal{A'}}(y)a)+{N_\mathcal{A}}^2({\rho_\mathcal{A'}}(x)b)
 +{N_\mathcal{A'}}^2({\rho_\mathcal{A}}(a)y)+{N_\mathcal{A'}}^2({\rho_\mathcal{A}}(b)x)
 -N_\mathcal{A}({\rho_\mathcal{A'}}(y){N_\mathcal{A}}(a))\\
 &&\hspace{-5mm} -N_\mathcal{A}({\rho_\mathcal{A'}}(N_\mathcal{A'}(x))b)-N_\mathcal{A'}({\rho_\mathcal{A}}(N_\mathcal{A}(a))y)
 -N_\mathcal{A'}({\rho_\mathcal{A}}(b)(N_\mathcal{A'}(x)))-N_\mathcal{A}({\rho_\mathcal{A'}}({N_\mathcal{A'} (y)})a) \\
 &&\hspace{-5mm} -N_\mathcal{A}({\rho_\mathcal{A'}}(x) N_\mathcal{A}(b))-N_\mathcal{A'}({\rho_\mathcal{A}}(a)N_\mathcal{A'}(y))-N_\mathcal{A'}({\rho_\mathcal{A'}}(N_\mathcal{A}(b))x),
 \end{eqnarray*}
 where $a,b\in \mathcal{A}, x,y\in \mathcal{A'}$.
 \end{proof}

\subsection{Manin triple of a Nijenhuis mock-Lie algebras} Recall that a bilinear form $\mathfrak{B}$ on a mock-Lie algebra $(\mathcal{A},[~,~])$ is called {\bf invariant} if, for all $x,y,z \in \mathcal{A}$,
 \begin{eqnarray}\label{eq:hi}
 \mathfrak{B}([x,y],z)=\mathfrak{B}(x,[y,z]).
 \end{eqnarray}

 \begin{pro}\label{pro:kl} Let $(\mathcal{A},[~,~],N)$ be a Nijenhuis mock-Lie algebra and $\mathfrak{B}$ be a nondegenerate invariant bilinear form on $(\mathcal{A},[~,~])$. Assume that $\widehat{N}$ is the adjoint linear map of $N$ with respect to $\mathfrak{B}$, characterized by
 \begin{eqnarray}\label{eq:fg}
 \mathfrak{B}(N(x),y)=\mathfrak{B}(x,\widehat{N}(y)), ~ \forall~x, y \in \mathcal{A}.
 \end{eqnarray}
 Then $\widehat{N}$ is adjoint-admissible to $(\mathcal{A},[~,~],N)$, or equivalently, $(\mathcal{A}^*,ad^*,{\widehat{N}}^*)$ is a representation of $(\mathcal{A},[~,~],N)$. Moreover, $(\mathcal{A}^*,ad^*,{\widehat{N}}^*)$ is equivalent to $(\mathcal{A},ad,N)$ as representations of $(\mathcal{A},[~,~],N)$. Conversely, let $(\mathcal{A},[~,~],N)$ be a Nijenhuis mock-Lie algebra and $S: \mathcal{A}\longrightarrow \mathcal{A}$ be a linear map that is adjoint-admissible to $(\mathcal{A},[~,~],N)$. If the resulting representation $(\mathcal{A}^*,ad^*,S^*)$ of $(\mathcal{A},[~,~],N)$ is equivalent to $(\mathcal{A},ad,N)$, then there exists a nondegenerate invariant bilinear form $\mathfrak{B}$ on $(\mathcal{A},[~,~],N)$ such that $S=\widehat{N}$.
 \end{pro}

 \begin{proof} For all $x,y,z \in \mathcal{A}$, we have
 \begin{eqnarray*}
 0&\stackrel {(\ref{eq:bja})}{=}&\mathfrak{B}([N(x),N(y)]+{N^2}[x,y]-N[N(x),y]-N[x,N(y)],z)\\
 &\stackrel {(\ref{eq:hi})(\ref{eq:fg})}{=}&
 \mathfrak{B}(x,\widehat{N}([N(y),z])+[y,{\widehat{N}}^2 (z)]-\widehat{N}[y,\widehat{N}(z)]-[N(y),\widehat{N}(z)]). 
 \end{eqnarray*}
 Then 
 \begin{eqnarray*}
 \widehat{N}([N(y),z]))+[y,{\widehat{N}}^2 (z)]-\widehat{N}[y,\widehat{N}(z)]-[N(y),\widehat{N}(z)]=0.
 \end{eqnarray*}
 By Lemma \mref{lem:HO}, $(\mathcal{A}^*,ad^*,{\widehat{N}}^*)$ is a representation of  $(\mathcal{A},[~,~],N)$. Next, we define the linear map $\psi:\mathcal{A} \longrightarrow \mathcal{A}^*$ by
 \begin{eqnarray}\label{eq:ik}
 \psi(x)y:=\langle\psi(x),y\rangle:=\mathfrak{B}(x,y),
 \end{eqnarray}
 for all $x,y \in \mathcal{A} $. The nondegeneracy of $\mathfrak{B}$ gives the bijectivity of $\psi$. For all $x,y,z \in \mathcal{A}$, we have
 \begin{eqnarray*}
 &&\psi(ad(x)y)z\stackrel{(\ref{eq:ik})}{=}\mathfrak{B}([x,y],z)\stackrel{(\ref{eq:1-1})}{=}\mathfrak{B}([y,x],z)\stackrel {(\ref{eq:hi})}{=}\mathfrak{B}(y,[x,z])\\
 &&\quad=\mathfrak{B}(y,ad(x)z)\stackrel {(\ref{eq:ik})}{=}\langle\psi(y),ad(x)z\rangle=\langle ad^*(x)\psi(y),z\rangle=ad^*(x)\psi(y)z,
 \end{eqnarray*}
 and similarly, 
 \begin{eqnarray*}
 &&\psi(N(x))y\stackrel{(\ref{eq:ik})}{=}\mathfrak{B}(N(x),y)\stackrel {(\ref{eq:fg})}{=}\mathfrak{B}(x,\widehat{N}(y))\stackrel {(\ref{eq:ik})}{=}\langle\psi(x),\widehat{N} (y)\rangle=\langle{\widehat{N}}^*(\psi(x)),y\rangle={\widehat{N}}^*(\psi(x))y.
 \end{eqnarray*}
 Hence, $(\mathcal{A}^*,ad^*,{\widehat{N}}^*)$ is equivalent to $(\mathcal{A},ad,N)$ as representations of $(\mathcal{A},[~,~],N)$.
 
 Conversely, suppose that $\psi:\mathcal{A}\longrightarrow \mathcal{A}^*$ is the linear isomorphism giving the equivalence between $(\mathcal{A},ad,N)$ and $(\mathcal{A}^*,ad^*,S^*)$. Define a bilinear form $\mathfrak{B}$ on $\mathcal{A}$ by
 \begin{eqnarray*}
 \mathfrak{B}(x,y):=\langle\psi(x),y\rangle,
 \end{eqnarray*}
 for all $x,y \in \mathcal{A}$. Then a similar argument gives the Nijenhuis    algebra $(\mathcal{A},[~,~],N)$ and $\widehat{N}=S$.
 \end{proof}

 \begin{defi}[\cite{BCHM}]\label{de:cf} Let $(\mathcal{A}, [~,~])$ and $(\mathcal{A}^*, [~,~]_{\mathcal{A}^*})$ be two mock-Lie algebras. A {\bf Manin triple of a mock-Lie algebra associated to $(\mathcal{A}, [~,~])$ and $(\mathcal{A}^*, [~,~]_{\mathcal{A}^*})$} is a quadruple $((\mathcal{A}\oplus \mathcal{A}^*,[~,~]_\diamond),\mathcal{A},\mathcal{A}^*,\mathfrak{B}_d)$ such that $(\mathcal{A}, [~,~])$ and $(\mathcal{A}^*, [~,~]_{\mathcal{A}^*})$ are mock-Lie subalgebras of the mock-Lie algebra $(\mathcal{A} \oplus \mathcal{A}^*,[~,~]_\diamond)$, and the natural nondegenerate symmetric bilinear form $\mathfrak{B}_d$ on $(\mathcal{A}\oplus \mathcal{A}^*,[~,~]_\diamond)$ given by 
 \begin{eqnarray}\label{eq:lpo}
 \mathfrak{B}_d(x+a^*,y+b^*)=\langle a^*,y\rangle+\langle b^*,x\rangle,\quad x,y \in \mathcal{A}, a^*,b^* \in \mathcal{A}^*,
 \end{eqnarray}
 is invariant.
 \end{defi}

 Now we extend this notion to Nijenhuis mock-Lie algebras.
 \begin{defi}\label{de:cf} Let $(\mathcal{A}, [~,~], N)$ and $(\mathcal{A}^*, [~,~]_{\mathcal{A}^*}, S^*)$ be two Nijenhuis mock-Lie algebras. A {\bf Manin triple $((\mathcal{A}\oplus \mathcal{A}^*,[~,~]_\diamond,N+S^*)$, $(\mathcal{A},N),(\mathcal{A}^*,S^*),\mathfrak{B}_d)$ of a Nijenhuis mock-Lie algebra associated to $(\mathcal{A}, [~,~], N)$ and $(\mathcal{A}^*, [~,~]_{\mathcal{A}^*}, S^*)$} is a Manin triple $((\mathcal{A}\oplus \mathcal{A}^*, [~,~]_\diamond),\mathcal{A},\mathcal{A}^*,\mathfrak{B}_d)$ of a mock-Lie algebra associated to $(\mathcal{A}, [~,~])$ and $(\mathcal{A}^*, [~,~]_{\mathcal{A}^*})$ such that $((\mathcal{A}\oplus \mathcal{A}^*,[~,~]_\diamond),N+S^*)$ is a Nijenhuis mock-Lie algebra. 
 \end{defi}

 \begin{lem}\label{lem:wb} Let $((\mathcal{A}\oplus \mathcal{A}^*,[~,~]_\diamond,N+S^*),(\mathcal{A},N),(\mathcal{A}^*,S^*),\mathfrak{B}_d)$ be a Manin triple of a Nijenhuis mock-Lie algebra associated to $(\mathcal{A}, [~,~], N)$ and $(\mathcal{A}^*, [~,~]_{\mathcal{A}^*}, S^*)$.
 \begin{enumerate}[(1)]
 \item \label{it:c6} The adjoint $\widehat{N+S^*}$ of $N+S^*$ with respect to $\mathfrak{B}_d$ is $S+N^*$. Furthermore,  $S+N^*$ is adjoint-admissible to $(\mathcal{A}\oplus \mathcal{A}^*,[~,~]_\diamond,N+S^*)$.
 \item \label{it:c7} $S$ is adjoint-admissible to $(\mathcal{A},[~,~],N)$.
 \item \label{it:c5} $N^*$ is adjoint-admissible to $(\mathcal{A}^*,[~,~]_{\mathcal{A}^*},S^*)$.
 \end{enumerate}
 \end{lem}
 
 \begin{proof} \ref{it:c6} For all $x,y \in \mathcal{A}, a^*,b^* \in \mathcal{A}^*$, we have
 \begin{eqnarray*}
 \mathfrak{B}_d((N+S^*)(x+a^*),y+b^*)
 &=&\mathfrak{B}_d(N(x)+S^*(a^*),y+b^*)\\
 &\stackrel {(\ref{eq:lpo})}{=}&\langle S^*(a^*),y\rangle+\langle b^*,N(x)\rangle\\
 &=&\langle a^*,S(y)\rangle+\langle N^*(b^*),x\rangle\\
 &\stackrel {(\ref{eq:lpo})}{=}&\mathfrak{B}_d(x+a^*,(S+N^*)(y+b^*)).
 \end{eqnarray*}
 Hence $\widehat{N+S^*}=S+N^*$. Furthermore, by Proposition \mref{pro:kl}, $S+N^*$ is admissible to $(\mathcal{A}\oplus \mathcal{A}^*,[~,~]_\diamond,N+S^*)$.
 
 \ref{it:c7} By \ref{it:c6}, $S+N^*$ is adjoint-admissible to $(\mathcal{A}\oplus \mathcal{A}^*,[~,~]_\diamond,N+S^*)$. Then by Eq.(\ref{eq:fb}) and let $N=N+S^*$, $\beta=S+N^*$, for all $x,y \in \mathcal{A}$ and $a^*,b^* \in \mathcal{A}^*$, we obtain
 \begin{eqnarray*}
 &&(S+N^*)[(N+S^*)(x+a^*),y+b^*]+[x+a^*,(S+N^*)^2 (y+b^*)]\\
 &&-[(N+S^*)(x+a^*),(S+N^*)(y+b^*)]-(S+N^*)[x+a^*,(S+N^*)(y+b^*)]=0.
 \end{eqnarray*}
 Now taking $a^*=b^*=0$  in the equation above gives the admissibility of $S$ to $(\mathcal{A},[~,~],N)$.
 
 \ref{it:c5} Taking $x=y=0$ in the equation above gives the admissibility of $N^*$ to $(\mathcal{A}^*,[~,~]_{\mathcal{A}^*},S^*)$.
 \end{proof}

 \begin{lem}[\cite{BCHM}]\label{lem:io} Let $(\mathcal{A},[~,~])$ be a mock-Lie algebra. Suppose that there is a mock-Lie algebra structure $[~,~]_{\mathcal{A}^*}$ on its dual space $\mathcal{A}^*$. Then there is a Manin triple $((\mathcal{A}\oplus \mathcal{A}^*, [~,~]_\diamond),\mathcal{A},\mathcal{A}^*,\mathfrak{B}_d)$ of a mock-Lie algebra associated to $(\mathcal{A}, [~,~])$ and $(\mathcal{A}^*, [~,~]_{\mathcal{A}^*})$ if and only if $(\mathcal{A},\mathcal{A}^*,ad^*,Ad^*)$ is a matched pair of mock-Lie algebras.
 \end{lem}

 \begin{thm}\label{thm:hdg} Let $(\mathcal{A},[~,~],N)$ be a Nijenhuis mock-Lie algebra. Suppose that there is a Nijenhuis mock-Lie algebra structure $(\mathcal{A}^*,[~,~]_{\mathcal{A}^*},S^*)$ on its dual space $\mathcal{A}^*$. Then there is a Manin triple $((\mathcal{A}\oplus \mathcal{A}^*,[~,~]_\diamond,N+S^*),(\mathcal{A},N),(\mathcal{A}^*,S^*),\mathfrak{B}_d)$ of a Nijenhuis mock-Lie algebra associated to $(\mathcal{A}, [~,~], N)$ and $(\mathcal{A}^*, [~,~]_{\mathcal{A}^*}, S^*)$ if and only if $((\mathcal{A},N),(\mathcal{A}^*,S^*),ad^*,Ad^*)$ is a matched pair of Nijenhuis mock-Lie algebras.
 \end{thm}

 \begin{proof} $(\Longrightarrow)$ Let $((\mathcal{A}\oplus \mathcal{A}^*,[~,~]_\diamond,N+S^*),(\mathcal{A},N),(\mathcal{A}^*,S^*),\mathfrak{B}_d)$ be a Manin triple of a Nijenhuis mock-Lie algebra associated to $(\mathcal{A}, [~,~], N)$ and $(\mathcal{A}^*, [~,~]_{\mathcal{A}^*}, S^*)$. Then $((\mathcal{A}\oplus \mathcal{A}^*,[~,~]_\diamond),\mathcal{A},\mathcal{A}^*,\mathfrak{B}_d)$ is a Manin triple of a mock-Lie algebra associated to $(\mathcal{A}, [~,~])$ and $(\mathcal{A}^*, [~,~]_{\mathcal{A}^*})$. By Lemma \mref{lem:io}, $(\mathcal{A},\mathcal{A}^*,ad^*,Ad^*)$ is a matched pair of mock-Lie algebras. Furthermore, by Lemma \mref{lem:wb}, $(\mathcal{A}^*,ad^*,S^*)$ is a representation of $(\mathcal{A},[~,~],N)$ and $(\mathcal{A},Ad^*,N)$ is a representation of $(\mathcal{A}^*,[~,~]_{\mathcal{A}^*},S^*)$. Hence, $((\mathcal{A},N),(\mathcal{A}^*,S^*),ad^*,Ad^*)$ is a matched pair of Nijenhuis mock-Lie algebras.
 
 $(\Longleftarrow)$ If $((\mathcal{A},N),(\mathcal{A}^*,S^*),ad^*,Ad^*)$ is a matched pair of Nijenhuis mock-Lie algebras, then $(\mathcal{A},\mathcal{A}^*$, $ad^*,Ad^*)$ is a matched pair of mock-Lie algebras. By Lemma \mref{lem:io}, $((\mathcal{A}\oplus \mathcal{A}^*,[~,~]_\diamond),\mathcal{A},\mathcal{A}^*,\mathfrak{B}_d)$ is a Manin triple of a mock-Lie algebra associated to $(\mathcal{A}, [~,~])$ and $(\mathcal{A}^*, [~,~]_{\mathcal{A}^*})$. Furthermore, $(\mathcal{A}\oplus \mathcal{A}^*,[~,~]_\diamond, N+S^*)$ is a Nijenhuis mock-Lie algebra by Theorem \ref{thm:gp}. Hence, $((\mathcal{A}\oplus \mathcal{A}^*,[~,~]_\diamond,N+S^*),(\mathcal{A},N),(\mathcal{A}^*,S^*),\mathfrak{B}_d)$ is a Manin triple of a Nijenhuis mock-Lie algebra associated to $(\mathcal{A}, [~,~], N)$ and $(\mathcal{A}^*, [~,~]_{\mathcal{A}^*}, S^*)$.
 \end{proof}

 \subsection{Nijenhuis mock-Lie bialgebras} For a finite dimensional vector space $\mathcal{A}$, $(\mathcal{A}^*,[~,~]_{\mathcal{A}^*},S^*)$ is a Nijenhuis mock-Lie algebra if and only if $(\mathcal{A},\D,S)$ is a Nijenhuis mock-Lie coalgebra, where $\D:\mathcal{A}\longrightarrow \mathcal{A}\o \mathcal{A}$ is the linear dual $[~,~]_{\mathcal{A}^*}:\mathcal{A}^* \otimes \mathcal{A}^* \longrightarrow \mathcal{A}^*$, that is, for all $x\in \mathcal{A}, a^*,b^* \in \mathcal{A}^*$,
 \begin{eqnarray*}
 \langle\D(x),a^* \otimes b^*\rangle=\langle x,[a^*,b^*]_{\mathcal{A}^*}\rangle.
 \end{eqnarray*}
 Moreover, for a linear map $N:\mathcal{A}\longrightarrow \mathcal{A}$, the condition that $N^*$ is adjoint-admissible to the Nijenhuis mock-Lie algebra $(\mathcal{A}^*,[~,~]_{\mathcal{A}^*},S^*)$, that is, for all $a^*,b^* \in \mathcal{A}^*$,
 \begin{eqnarray*}
 &N^*([S^*(a^*),b^*])+[a^*,{N^*}^2(b^*)]=[S^*(a^*),N^*(b^*)]+N^*([a^*,N^*(b^*)]),&
 \end{eqnarray*}
can be rewritten in terms of $\D$ as
 \begin{eqnarray}\label{eq:jy}
 &(S \otimes \id)\D N+(\id \otimes N^2)\D=(S\otimes N)\D+(\id \otimes N)\D N.&
 \end{eqnarray}

 \begin{defi}\label{de:yy} A Nijenhuis mock-Lie bialgebra is a vector space $\mathcal{A}$  together with linear maps $[~,~]: \mathcal{A}\otimes \mathcal{A}\longrightarrow\mathcal{A}$, $\D : \mathcal{A}\longrightarrow \mathcal{A}\otimes \mathcal{A}$, $N, S: \mathcal{A}\longrightarrow \mathcal{A}$ such that
 \begin{enumerate}[(1)]
 \item  $(\mathcal{A},[~,~],\D)$ is a mock-Lie bialgebra.
 \item  $(\mathcal{A},[~,~],N)$ is a Nijenhuis  mock-Lie algebra.
 \item  $(\mathcal{A},\D,S)$ is a Nijenhuis mock-Lie coalgebra.
 \item  $S$ is adjoint-admissible to $(\mathcal{A},[~,~],N)$, that is, Eq.(\ref{eq:sfh}) holds.
 \item  $N^*$ is adjoint-admissible to $(\mathcal{A}^*,\D^*,S^*)$, that is, Eq.(\ref{eq:jy}) holds.
 \end{enumerate}
 We denote the Nijenhuis mock-Lie bialgebra by $(\mathcal{A},[~,~],N,\D,S)$.
 \end{defi}

 \begin{ex}\label{ex:ng} Let $(\mathcal{A},[~,~])$ be a mock-Lie algebra with basis $\mathcal{B}=\{e_1,e_2,e_3,e_4\}$ and the product $[~,~]$ defined by
 \begin{center}
        \begin{tabular}{r|rrrr}
          $[~,~]$ & $e_1$  & $e_2$ & $e_3$ & $e_4$  \\
          \hline
           $e_1$ & $e_2$  & $0$  & $e_4$ & $0$\\
           $e_2$ & $0$  &  $0$  & $0$ & $0$\\
           $e_3$ & $e_4$ & $0$ & $0$ & $0$\\
           $e_4$ & $0$ & $0$ & $0$ & $0$\\
        \end{tabular}.
        \end{center}
        Define
        $$\left\{
            \begin{array}{l}
             N(e_1)=e_1\\
             N(e_2)=e_2+e_3\\
             N(e_3)=e_3\\
             N(e_4)=e_4\\
            \end{array}
            \right.. $$
     Then $(\mathcal{A},[~,~],N)$ is a Nijenhuis mock-Lie algebra. Set
        $$\left\{
            \begin{array}{l}
            \D(e_1)=-e_4\o e_2-e_2\o e_4\\
             \D(e_2)=0\\
             \D(e_3)=0\\
             \D(e_4)=0\\
            \end{array}
            \right.$$
        and 
      $$\left\{
            \begin{array}{l}
             S(e_1)=e_1+\l e_3+\g e_4\\
             S(e_2)=e_2-e_3\\
             S(e_3)=e_3\\
             S(e_4)=e_4\\
            \end{array}
            \right.\quad (\l, \g~\hbox{are parameters}). $$
      Then $(\mathcal{A},\D,S)$ is a Nijenhuis mock-Lie coalgebra.\\ Furthermore, $(\mathcal{A},[~,~],\D,N,S)$ is a Nijenhuis mock-Lie bialgebra.
 \end{ex}

 \begin{lem}[{\cite{BCHM}}]\label{lem:jn}
 Let $(\mathcal{A},[~,~])$ be a mock-Lie algebra. Suppose that there is a mock-Lie algebra structure $[~,~]_{\mathcal{A}^*}$ on its dual space $\mathcal{A}^*$. Then the triple $(\mathcal{A},[~,~],\D)$ is a mock-Lie bialgebra if and only if $(\mathcal{A},\mathcal{A}^*,ad^*,Ad^*)$  is a matched pair of mock-Lie algebras.
 \end{lem}

 \begin{thm}\label{thm:hup} Let $(\mathcal{A},[~,~],N)$ be a Nijenhuis mock-Lie algebra. Suppose that there is a  Nijenhuis mock-Lie algebra structure $(\mathcal{A}^*,[~,~]_{\mathcal{A}^*},S^*)$ on its dual space $\mathcal{A}^*$. Then the quintuple $(\mathcal{A},[~,~],N,\D,S)$ is a Nijenhuis mock-Lie bialgebra if and only if $((\mathcal{A},N),(\mathcal{A}^*,S^*),ad^*,Ad^*)$ is a matched pair of Nijenhuis mock-Lie algebras.
 \end{thm}

 \begin{proof}$(\Longrightarrow)$ If $(\mathcal{A},[~,~],N,\D,S)$ is a Nijenhuis mock-Lie bialgebra, then by Definition \mref{de:yy}, we know that $(\mathcal{A},[~,~],\D)$ is a mock-Lie bialgebra and $S$, $N^*$ are admissible to $(\mathcal{A},[~,~],N)$ and $(\mathcal{A}^*,[~,~]_{\mathcal{A}^*}$, $S^*)$ respectively. It means that $(\mathcal{A}^*,ad^*,S^*)$ is a representation of $(\mathcal{A},[~,~],N)$ and $(\mathcal{A},Ad^*,N)$ is a representation of $(\mathcal{A}^*,[~,~]_{\mathcal{A}^*},S^*)$. By Lemma \mref{lem:jn}, $(\mathcal{A},\mathcal{A}^*,ad^*,Ad^*)$ is a matched pair of mock-Lie algebras. Therefore, $((\mathcal{A},N),(\mathcal{A}^*,S^*),ad^*,Ad^*)$ is a matched pair of Nijenhuis mock-Lie algebras.
 
 $(\Longleftarrow)$ If $((\mathcal{A},N),(\mathcal{A}^*,S^*),ad^*,Ad^*)$ is a matched pair of Nijenhuis mock-Lie algebras, then by Definition \mref{de:mv}, $(\mathcal{A},\mathcal{A}^*,ad^*,Ad^*)$ is a matched pair of mock-Lie algebras and $S$, $N^*$ are admissible to $(\mathcal{A},[~,~],N)$ and $(\mathcal{A}^*,[~,~]_{\mathcal{A}^*},S^*)$ respectively. Hence, by Lemma \mref{lem:jn}, $(\mathcal{A},[~,~],N,\D,S)$ is a Nijenhuis mock-Lie bialgebra.
 \end{proof}

 \begin{thm}\label{thm:pog} Let $(\mathcal{A},[~,~],N)$ and $(\mathcal{A}^*,\D^*,S^*)$ be two Nijenhuis mock-Lie algebras. Then the following conditions are equivalent:
 \begin{enumerate}[(1)]
 \item $((\mathcal{A},N),(\mathcal{A}^*,S^*),ad^*,Ad^*)$ is a matched pair of Nijenhuis mock-Lie algebras.
 \item There is a Manin triple $((\mathcal{A}\oplus \mathcal{A}^*,[~,~]_\diamond,N+S^*),(\mathcal{A},N),(\mathcal{A}^*,S^*),\mathfrak{B}_d)$ of a Nijenhuis mock-Lie algebra associated to $(\mathcal{A}, [~,~], N)$ and $(\mathcal{A}^*, [~,~]_{\mathcal{A}^*}, S^*)$.
 \item  $(\mathcal{A},[~,~],\D,N,S)$ is a Nijenhuis mock-Lie bialgebra.
 \end{enumerate}
 \end{thm}

 \begin{proof}
 It is obvious by Theorem \mref{thm:hdg} and Theorem \mref{thm:hup}.
 \end{proof}

\subsection{Coboundary Nijenhuis mock-Lie bialgebras}
 \begin{defi}\label{de:bh}
 A Nijenhuis mock-Lie bialgebra $(\mathcal{A},[~,~],\D,N,S)$ is called {\bf coboundary} if there exists  $r \in\mathcal{A} \otimes \mathcal{A}$ such that Eq.(\ref{eq:cop}) holds.
 \end{defi}

 \begin{rmk}[\cite{BCHM}]\label{rmk:fm} Let $(\mathcal{A},[~,~])$ be a mock-Lie algebra and $r=\sum\limits_{i} a_i\o b_i \in\mathcal{A}\otimes \mathcal{A}$. Then the map $\D: \mathcal{A}\longrightarrow \mathcal{A} \otimes \mathcal{A}$ defined by Eq.(\ref{eq:cop}) induces a mock-Lie algebra structure on $\mathcal{A}^*$ such that $(\mathcal{A},[~,~],\D)$ is a mock-Lie bialgebra if and only if, for all $x\in \mathcal{A}$,
 \begin{eqnarray}
 &(ad(x)\otimes \id-\id\otimes ad(x))(r+\tau(r))=0,&\label{eq:jio}\\
 &(ad(x)\otimes \id\otimes \id+\id \otimes ad(x)\otimes \id+\id \otimes \id \otimes ad(x))([r_{12},r_{13}]
 +[r_{13},r_{23}]-[r_{12},r_{23}])=0,\quad&\label{eq:jkp}
 \end{eqnarray}
 where $[r_{12},r_{13}]=\sum\limits_{i,j}[a_i , a_j]\otimes b_i \otimes b_j$, $[r_{13},r_{23}]=\sum\limits_{i,j} a_i \otimes a_j\otimes[ b_i , b_j]$, $[r_{12},r_{23}]=\sum\limits_{i,j} a_i \otimes [b_i, a_j] \otimes b_j$.
\end{rmk}

 Suppose that $(\mathcal{A},[~,~],N)$ is an $S$-adjoint-admissible Nijenhuis mock-Lie algebra. To make  $(\mathcal{A},[~,~],\D,N,S)$  a Nijenhuis mock-Lie bialgebra, we only need that $({\mathcal{A}}^*,\D^*,S^*)$ is an $N^*$-adjoint-admissible Nijenhuis mock-Lie algebra, that is, $(\mathcal{A},\D,S)$ is a Nijenhuis Lie coalgebra and Eq.(\ref{eq:jy}) holds.

 \begin{pro}\label{pro:nh} Let $(\mathcal{A},[~,~],N)$ be an $S$-adjoint-admissible Nijenhuis mock-Lie algebra and $r\in \mathcal{A} \otimes \mathcal{A}$. Define a linear map $\D: \mathcal{A}\longrightarrow \mathcal{A} \otimes \mathcal{A}$  by Eq.(\ref{eq:cop}). Then the following assertions hold.
 \begin{enumerate}[(1)]
 \item \label{it:0} Eq.(\ref{eq:gd}) holds if and only if, for all $x\in \mathcal{A}$,
 \begin{eqnarray}\label{eq:mg}
 &&(\id \otimes ad(S(x))-\id \otimes S\circ ad(x))(S\otimes\id -\id \otimes N)(r)\nonumber\\
 &&+(ad(S(x))\otimes \id-S\circ ad(x)\otimes \id)(N\otimes \id-\id \otimes S)(r)=0.
 \end{eqnarray}
 \item  \label{it:7} Eq.(\ref{eq:jy}) holds if and only if, for all $x\in \mathcal{A}$,
 \begin{eqnarray}\label{eq:jkw}
 &&(S\circ ad(x)\otimes \id+ad(N(x))\otimes \id-\id \otimes ad(N(x)) +\id \otimes N\circ ad(x))(S\otimes \id-\id \otimes N)(r)\nonumber\\
 &&+(ad(x)\o N^2-ad(x)\circ S^2 \otimes \id)(r)=0.
 \end{eqnarray}
 \end{enumerate}
 \end{pro}

 \begin{proof} \ref{it:0} For all $r=\sum a_i\otimes b_i $ and $x\in \mathcal{A}$, we have
 \begin{eqnarray*}
 &&\hspace{-20mm}(S \otimes S)\D(x)+\D(S^2 (x))-(S \otimes \id)\D(S(x))-(\id \otimes S)\D(S(x))\\
 &\stackrel {(\ref{eq:cop})(\ref{eq:1-1})(\ref{eq:sfh})}{=}&(S \otimes S)([x,a_i]\otimes b_i)-(S \otimes S)(a_i \otimes [x,b_i])+(S\otimes \id)(a_i\otimes [S(x),b_i])\\
 &&-(\id \otimes S)([S(x),a_i]\otimes b_i)+([N(a_i),S(x)]-[S(N(a_i)),x])\otimes b_i\\
 &&+a_i\otimes ([S(N(b_i)),x]-[N(b_i),S(x)])\\ 
 &=&\hspace{-6mm}(\id \otimes ad(S(x))-\id \otimes S\circ ad(x))(S\otimes \id -\id \otimes N)(r)+(ad(S(x))\otimes \id-S\circ ad(x)\otimes \id)\\
 &&(N\otimes \id-\id\otimes S)(r).
 \end{eqnarray*}
 Then, Eq.(\ref{eq:gd}) holds if and only if Eq.(\ref{eq:mg}) holds.

 \ref{it:7} Similarly, we have
 \begin{eqnarray*}
 &&\hspace{-20mm}(S \otimes \id)\D N(x)+(\id \otimes N^2)\D(x)-(S\otimes N)\D(x)-(\id\otimes N)\D N(x)\\
 &\stackrel {(\ref{eq:cop})(\ref{eq:bja})(\ref{eq:sfh})}{=}&(S([x,S(a_i)])+[N(x),S(a_i)]-[x,S^2 (a_i)])\otimes b_i-S(a_i)\otimes [N(x),b_i]+[x,a_i]\otimes N^2(b_i)\\
 &&+a_i\otimes ([N(x),N(b_i)]-N([x,N(b_i)]))-S([x,a_i])\otimes N(b_i)+S(a_i)\otimes N([x,b_i])\\
 &&-[N(x),a_i]\otimes N(b_i)\\ 
 &=&\hspace{-6mm}(S\circ ad(x)\otimes \id+ad(N(x))\otimes \id-\id \otimes ad(N(x)) +\id \otimes N\circ ad(x))(S\otimes \id-\id \otimes N)(r)\\
 &&+(ad(x)\o N^2-ad(x)\circ S^2 \otimes \id)(r).
 \end{eqnarray*} 
 Then, Eq.(\ref{eq:jy}) holds if and only if Eq.(\ref{eq:jkw}) holds.
 \end{proof}

 \begin{rmk}\label{rmk:q}
 If $(S\otimes\id-\id\otimes N)(r)=0$, then $(ad(x)\o N^2-ad(x)\circ S^2 \otimes\id)(r)=0$.
 \end{rmk} 

 \begin{thm}\label{thm:ggg} Let $(\mathcal{A},[~,~],N)$ be an $S$-adjoint-admissible Nijenhuis mock-Lie algebra and $r\in \mathcal{A} \otimes \mathcal{A}$. Define a linear map $\D$ by Eq.(\ref{eq:cop}). Then $(\mathcal{A},[~,~],\D,N,S)$ is a Nijenhuis mock-Lie bialgebra if and only if Eqs.(\ref{eq:jio})-(\ref{eq:jkw}) hold.
 \end{thm}

 \begin{proof} The tuple $(\mathcal{A},[~,~],\D,N,S)$  defines a Nijenhuis mock-Lie bialgebra if and only if $(\mathcal{A}^*,\D^*,S^*)$ is an $N^*$-adjoint-admissible Nijenhuis mock-Lie algebra. By Proposition \mref{pro:nh}, the latter holds if and only if Eqs.(\ref{eq:jio})-(\ref{eq:jkw}) hold.
 \end{proof}

 \begin{cor}\label{cor:gm} Let $(\mathcal{A},[~,~],N)$ be an $S$-adjoint-admissible Nijenhuis mock-Lie algebra and $r \in \mathcal{A}\otimes \mathcal{A}$. Define a linear map $\D: \mathcal{A}\longrightarrow \mathcal{A}\otimes \mathcal{A}$ by Eq.(\ref{eq:cop}). Then $(\mathcal{A},[~,~],\D,N,S)$ is a Nijenhuis mock-Lie bialgebra if Eq.(\ref{eq:jio}) and the following equations hold:
 \begin{eqnarray}
 &[r_{12},r_{13}]+[r_{13},r_{23}]-[r_{12},r_{23}]=0,&\label{eq:uj}\\
 &(N\otimes \id-\id\otimes S)(r)=0,&\label{eq:cg}\\
 &(S\otimes \id-\id\otimes N)(r)=0.&\label{eq:ma}
 \end{eqnarray}
 \end{cor}
 We note that if $r$ is anti-symmetric in the sense of $r=-\tau(r)$, then Eq.(\ref{eq:cg}) holds if and only if Eq.(\ref{eq:ma}) holds.

 \begin{defi}\label{de:nj} Let $(\mathcal{A},[~,~],N)$ be a Nijenhuis mock-Lie algebra. Suppose that $r\in \mathcal{A}\otimes \mathcal{A}$ and $S:\mathcal{A}\longrightarrow \mathcal{A}$ is a linear map. Then,  Eq.(\ref{eq:uj}) together with Eqs.(\ref{eq:cg}) and $(\ref{eq:ma})$  is called an {\bf $S$-admissible mock-Lie-Yang-Baxter equation} (abbr. mLYBe) in $(\mathcal{A},[~,~],N)$.
 \end{defi}

 \begin{pro}\label{pro:gth} Let $(\mathcal{A},[~,~],N)$  be an $S$-adjoint-admissible Nijenhuis mock-Lie algebra and $r\in \mathcal{A}\otimes \mathcal{A}$ be an anti-symmetric solution of the $S$-admissible mLYBe in $(\mathcal{A},[~,~],N)$. Then $(\mathcal{A},[~,~],\D,N,S)$ is a coboundary Nijenhuis mock-Lie bialgebra, where the linear map $\D=\D_r$ is defined by Eq.(\ref{eq:cop}).
 \end{pro}

 \begin{proof} It is obvious by Corollary \mref{cor:gm} and Definition \mref{de:nj}.
 \end{proof}

 \begin{ex}\label{ex:ng-1} Let $(\mathcal{A},[~,~],\D,N,S)$ be the Nijenhuis mock-Lie bialgebra given in Example \ref{ex:ng}. Then $(\mathcal{A},[~,~],\D,N,S)$ is coboundary with $r= e_2 \o e_3- e_3 \o e_2$.  
 \end{ex}

\subsection{$O$-operators on Nijenhuis mock-Lie algebras}

 For a  vector space $V$, the isomorphism $V\otimes V \cong Hom(V^*, V)$ identifies an $r\in V\otimes V$ with a linear map $T_r: V^*\longrightarrow V$. Thus for $r=\sum\limits_{i} u_i\otimes v_i$, the corresponding map $T_r$ is
 \begin{eqnarray}\label{eq:nm}
 T_r(u^*)=\sum\limits_{i} \langle u^*, u_i\rangle v_i,\quad \forall u^*\in V^*.
 \end{eqnarray} 

 \begin{pro}[\cite{BCHM}]\label{pro:gxc} Let $(\mathcal{A},[~,~])$  be a mock-Lie algebra and $r \in \mathcal{A}\otimes \mathcal{A}$ be anti-symmetric. Then $r$ is a solution of the mLYBe in $(\mathcal{A},[~,~])$ if and only if $T_r$ satisfies
 \begin{eqnarray*}
 [T_r(x^*),T_r(y^*)]=T_r(ad^*(T_r(x^*))y^*+ad^*(T_r(y^*))x^*), \forall x^*, y^*\in \mathcal{A}^*.
 \end{eqnarray*}
 \end{pro}

 \begin{thm}\label{thm:x} Let $(\mathcal{A},[~,~],N)$ be a Nijenhuis mock-Lie algebra and $r\in \mathcal{A}\otimes \mathcal{A}$ be anti-symmetric. Let $S: \mathcal{A}\longrightarrow \mathcal{A}$ be a linear map. Then $r$ is a solution of the $S$-admissible mLYBe in $(\mathcal{A},[~,~],N)$ if and only if $T_r$ satisfies
 \begin{eqnarray}
 &[T_r(x^*),T_r(y^*)]=T_r(ad^*(T_r(x^*))y^*+ad^*(T_r(y^*))x^*), \forall x^*,y^*\in \mathcal{A}^*,&\label{eq:lq}\\
 &N\circ T_r=T_r\circ S^*.&\label{eq:lu}
 \end{eqnarray}
 \end{thm}

 \begin{proof} By Proposition \mref{pro:gxc}, $r$ is a solution of the mLYBe in $(\mathcal{A},[~,~])$ if and only if Eq.(\ref{eq:lq}) holds. Moreover, let $r=\sum\limits_{i} a_i \otimes b_i$ and for all $x^* \in \mathcal{A}^*$, we have
 \begin{eqnarray*}
 T_r(S^*(x^*))\stackrel {(\ref{eq:nm})}{=}\sum\limits_{i}\langle S^*(x^*),a_i\rangle b_i=\sum_i\langle x^*,S(a_i)\rangle b_i,
 \end{eqnarray*}
 \begin{eqnarray*}
 N(T_r(x^*))\stackrel {(\ref{eq:nm})}{=}\sum\limits_{i}\langle x^*,a_i\rangle N(b_i).
 \end{eqnarray*}
 Hence, $N\circ T_r=T_r\circ S^*$ if and only if Eq.(\ref{eq:ma}) holds. Hence, the conclusion follows.
 \end{proof}

 According to Theorem \ref{thm:x}, we introduce the following notion of weak $O$-operator.
 
 \begin{defi}\label{de:hc} Let $(\mathcal{A},[~,~],N)$ be a Nijenhuis mock-Lie algebra, $(V,\rho)$ be a representation of $(\mathcal{A},[~,~])$ and $\alpha: V\longrightarrow V$ be a linear map. A linear map $T:V\longrightarrow \mathcal{A}$ is called a {\bf weak $O$-operator associated to $(V,\rho)$ and $\alpha$} if $T$ satisfies
 \begin{eqnarray}
 &[T(u),T(v)]=T(\rho(T(u))v+\rho(T(v))u), \forall u, v \in V,&\label{eq:lz}\\
 &N\circ T=T\circ \alpha.&\label{eq:lm}
 \end{eqnarray}
 If $(V,\rho,\alpha)$ is a representation of $(\mathcal{A},[~,~],N)$, then $T$ is called an {\bf $O$-operator associated to $(V,\rho,\alpha)$}.
 \end{defi}

Theorem \mref{thm:x} can be rewritten in terms of $O$-operators as follows.

 \begin{cor}\label{cor:bbo} Let $(\mathcal{A},[~,~],N)$ be a Nijenhuis mock-Lie algebra, $r\in \mathcal{A}\otimes \mathcal{A}$ be anti-symmetric and $S:\mathcal{A}\longrightarrow \mathcal{A}$ be a linear map. Then $r$ is a solution of the $S$-admissible mLYBe in $(\mathcal{A},[~,~],N)$ if and only if $T_r$ is a weak $O$-operator associated to $(\mathcal{A}^*,Ad^*)$ and $S^*$. In addition, if $(\mathcal{A},[~,~],N)$ is an $S$-adjoint-admissible Nijenhuis mock-Lie algebra, then $r$ is a solution of the $S$-admissible mLYBe in $(\mathcal{A},[~,~],N)$ if and only if $T_r$ is an $O$-operator associated to $(\mathcal{A}^*,Ad^*,S^*)$.
 \end{cor}

 \begin{lem}[\cite{Rp}]\label{lem:ri} Let $(\mathcal{A},[~,~])$ be a mock-Lie algebra, $(V,\rho)$ be a representation of $(\mathcal{A},[~,~])$ and $T:V\longrightarrow \mathcal{A}$ be a linear map which is identified to an element in $(\mathcal{A}\ltimes_{\rho^*}V^*)\otimes (\mathcal{A}\ltimes_{\rho^*}V^*)$ through the isomorphism  $Hom(V,\mathcal{A})\cong\mathcal{A}\otimes V^*\subseteq (\mathcal{A}\ltimes_{\rho^*}V^*)\otimes (\mathcal{A}\ltimes_{\rho^*}V^*)$. Then $r=T-\tau(T)$ is an anti-symmetric solution of the mLYBe in the semi-direct product mock-Lie algebra $\mathcal{A}\ltimes_{\rho^*}V^*$ if and only if $T$ is an $O$-operator of $(\mathcal{A},[~,~])$ associated to $(V,\rho)$.
 \end{lem}

 \begin{thm}\label{thm:fv} Let $(\mathcal{A},[~,~],N)$ be a Nijenhuis mock-Lie algebra, $(V,\rho)$ be a representation of $(\mathcal{A},[~,~])$, $S : \mathcal{A}\longrightarrow \mathcal{A}$ and $\alpha, \beta:V\longrightarrow V$ be linear maps. Then the following conditions are equivalent.
 \begin{enumerate}[(1)]
 \item \label{it:e2} There is a Nijenhuis mock-Lie algebra $(\mathcal{A} \ltimes_\rho V, N+\alpha)$ such that the linear map $S+\beta$ on $\mathcal{A}\oplus V$ is adjoint-admissible to $(\mathcal{A} \ltimes_\rho V, N+\alpha)$.
 \item \label{it:e3} There is a Nijenhuis mock-Lie algebra $(\mathcal{A} \ltimes_{\rho^*} V^*, N+\beta^*)$ such that the linear map $S+\alpha^*$ on $\mathcal{A}\oplus V^*$ is adjoint-admissible to $(\mathcal{A} \ltimes_{\rho^*} V^*, N+\beta^*)$.
 \item \label{it:e1} The following conditions are satisfied:
 \begin{enumerate}[(a)]
 \item \label{it:a1} $(V, \rho, \alpha)$ is a representation of $(\mathcal{A}, [~,~], N)$, that is, Eq.(\ref{eq:1}) holds.
 \item \label{it:a2} $S$ is adjoint-admissible to $(\mathcal{A}, [~,~], N)$, that is, Eq.(\ref{eq:sfh}) holds.
 \item \label{it:a3} $\beta$ is admissible to $(\mathcal{A}, [~,~], N)$ on $(V, \rho)$, that is, Eq.(\ref{eq:fb}) holds.
 \item \label{it:a4} The following equation holds:
 \begin{eqnarray}\label{eq:hjl}
 &\beta(\rho(y)\alpha(u))+\rho(S^2(y))u=\rho(S(y))\alpha(u)+\beta(\rho(S(y))u), \forall y\in \mathcal{A}, u\in V.&
 \end{eqnarray}
 \end{enumerate}
 \end{enumerate}
 \end{thm}

 \begin{proof} $\ref{it:e2}\Longleftrightarrow\ref{it:e1}$ By Proposition \mref{pro:B}, we have $(\mathcal{A} \ltimes_\rho V,[~,~]_\star, N+\alpha)$ is a Nijenhuis mock-Lie algebra if and only if $(V, \rho, \alpha)$ is a representation of $(\mathcal{A},[~,~],N)$. Let $x,y\in \mathcal{A}$ and $u,v\in V$, then we have
 \begin{eqnarray*}
 (S+\beta)([(N+\alpha)(x+u),y+v])
 &\stackrel{(\ref{eq:d})(\ref{eq:m})}{=}&(S+\beta)([N(x),y]+\rho(N(x))v+\rho(y)\alpha(u))\\
 &\stackrel{(\ref{eq:d})}{=}&S([N(x),y])+\beta(\rho(N(x))v+\rho(y)\alpha(u)),
 \end{eqnarray*}
 \begin{eqnarray*}
 [x+u,(S+\beta)^2(y+v)]\stackrel{(\ref{eq:d}(\ref{eq:m})}{=}[x,S^2(y)]+\rho(x)\beta^2(v)+\rho(S^2(y))u,
 \end{eqnarray*}
 \begin{eqnarray*}
 [(N+\alpha)(x+u),(S+\beta)(y+v)]\stackrel{(\ref{eq:d})(\ref{eq:m})}{=}[N(x),S(y)]+\rho(N(x))\beta(v)+\rho(S(y))\alpha(u),
 \end{eqnarray*}
 \begin{eqnarray*}
 (S+\beta)([x+u,(S+\beta)(y+v)])
 \stackrel{(\ref{eq:d})(\ref{eq:m})}{=}S([x,S(y)])+\beta(\rho(x)\beta(v)+\rho(S(y))u).
 \end{eqnarray*}
 We obtain that
 \begin{eqnarray}\label{eq:cbv}
 &&S([N(x),y])+\beta(\rho(N(x))v+\rho(y)\alpha(u))+[x,S^2(y)]+\rho(x)\beta^2(v)+\rho(S^2(y))u-[N(x),S(y)]\nonumber \\
 &&-\rho(N(x))\beta(v)-\rho(S(y))\alpha(u)-S([x,S(y)])-\beta(\rho(x)\beta(v)-\rho(S(y))u)=0.
 \end{eqnarray}
 Then, Eq.(\ref{eq:cbv}) holds if and only if Eq.(\ref{eq:sfh}) (corresponding to $u=v=0$), Eq.(\ref{eq:fb}) (corresponding to $u=y=0$) and Eq.(\ref{eq:hjl}) (corresponding to  $x=v=0$) hold. Hence, \ref{it:e2} holds if and only if \ref{it:e1} holds.

 $\ref{it:e3}\Longleftrightarrow\ref{it:e1}$ In \ref{it:e2}, take
 \begin{eqnarray*}
 V=V^*, \rho=\rho^*, \alpha=\beta^*, \beta=\alpha^*.
 \end{eqnarray*}
 Then from the above equivalence between \ref{it:e2} and \ref{it:e1}, we have \ref{it:e3} holds if and only if \ref{it:a1}-\ref{it:a3} hold and the following equation holds:
 \begin{eqnarray}\label{eq:cse}
 \alpha^*(\rho^*(y)\beta^*(v^*))+\rho^*(S^2(y))v^*-\rho^*(S(y))\beta^*(v^*)-\alpha^*(\rho^*(S(y))v^*)=0, \forall y\in \mathcal{A}, v^*\in V^*.
 \end{eqnarray}
 For all $y \in \mathcal{A}$, $u\in V$ and $v^* \in V^*$, we have
 \begin{eqnarray*}
 \langle\alpha^*(\rho^*(y)\beta^*(v^*)),u\rangle \stackrel{(\ref{eq:sr})}{=}\langle\rho^*(y)\beta^*(v^*),\alpha(u)\rangle=\langle\beta^*(v^*),\rho(y)\alpha(u)\rangle=\langle v^*,\beta(\rho(y)\alpha(u))\rangle,
 \end{eqnarray*}
 \begin{eqnarray*}
 \langle\rho^*(S^2(y))v^*,u\rangle\stackrel {(\ref{eq:sr})}{=}\langle v^*,\rho(S(y))u\rangle,
 \end{eqnarray*}
 \begin{eqnarray*}
 \langle\rho^*(S(y))\beta^*(v^*),u\rangle\stackrel {(\ref{eq:sr})}{=}\langle\beta^*(v^*),\rho(S(y)u)\rangle=\langle v^*,\beta(\rho(S(y)u))\rangle,
 \end{eqnarray*}
 \begin{eqnarray*}
 \langle\alpha^*\rho^*(S(y))v^*),u\rangle\stackrel {(\ref{eq:sr})}{=}\langle\rho^*S((y))v^*,\alpha(u)\rangle=\langle v^*,\rho(S((y))\alpha(u)\rangle.
 \end{eqnarray*}
 One can get that
 \begin{eqnarray*}
 &&\hspace{-16mm}\langle\alpha^*(\rho^*(y)\beta^*(v^*))+\rho^*(S^2(y))v^*-\rho^*(S(y))\beta^*(v^*)-\alpha^*(\rho^*(S(y))v^*),u\rangle\\
 &\stackrel {(\ref{eq:sr})}{=}&\langle v^*,\beta (\rho(y)\alpha(u))+\rho(S^2(y))u-\rho(S(y))\alpha(u)-\beta(\rho(S(y))u)\rangle.
 \end{eqnarray*}
 Hence, Eq.(\ref{eq:cse}) holds if and only if Eq.(\ref{eq:hjl}) holds. Therefore, \ref{it:e3} holds if and only if \ref{it:e1} holds.
 \end{proof}

 \begin{thm}\label{thm:hnv} Let $(\mathcal{A},[~,~],N)$ be a Nijenhuis mock-Lie algebra,  $(V,\rho)$ be a representation of $(\mathfrak{\mathcal{A}},[~,~])$, $(V^*, \rho^*,\beta^*)$ be a representation of $(\mathcal{A},[~,~],N)$, $S:\mathcal{A}\longrightarrow \mathcal{A}$, $\alpha:V\longrightarrow V$ and $T:V\longrightarrow \mathcal{A}$ be linear maps.
 \begin{enumerate}[(1)]
 \item \label{it:xn} $r=T-\tau(T)$ is an anti-symmetric solution of the $(S+\alpha^*)$-admissible mLYBe in the Nijenhuis mock-Lie algebra $(\mathcal{A}\ltimes_{\rho^*}V^*, N+\beta^*)$ if and only if $T$ is a weak $O$-operator associated to $(V,\rho)$ and $\alpha$, and satisfies $T\circ \beta=S\circ T$.
 \item \label{it:r6} Assume that $(V,\rho,\alpha)$ is a representation of $(\mathcal{A},[~,~],N)$. If $T$ is a weak $O$-operator associated to $(V,\rho,\alpha)$ and $T\circ \beta=S\circ T$, then $r=T-\tau(T)$ is an anti-symmetric solution of the $(S+\alpha^*)$-admissible mLYBe in the Nijenhuis mock-Lie algebra $(\mathfrak{\mathcal{A}}\ltimes_{\rho^*}V^*, N+\beta^*)$. In addition, if $(\mathcal{A},[~,~],N)$ is  $S$-adjoint-admissible and Eq.(\ref{eq:hjl}) holds such that the Nijenhuis mock-Lie algebra $(\mathcal{A}\ltimes_{\rho^*}V^*, N+\beta^*)$ is $(S+\alpha^*)$-adjoint-admissible, then there is a Nijenhuis mock-Lie bialgebra $(\mathcal{A}\ltimes_{\rho^*}V^*, N+\beta^*,\D, S+\alpha^*)$, where the linear map $\D=\D_r$ is defined by Eq.(\ref{eq:cop}) with $r=T-\tau(T)$.
 \end{enumerate}
 \end{thm}

 \begin{proof} \ref{it:xn} By Lemma \mref{lem:ri}, $r$ satisfies Eq.(\ref{eq:uj}) if and only if Eq.(\ref{eq:lz}) holds. Let $\{e_1,e_2,\cdots, e_n\}$ be a basis of $V$ and $\{e_1^*,e_2^*,\cdots, e_n^*\}$ be the dual basis. Then $T=\sum_{i=1}^n T(e_i)\otimes e_i^*\in (\mathcal{A}\ltimes_{\rho^*}V^*)\otimes \mathcal{A}\ltimes_{\rho^*}V^*$. Hence
 \begin{eqnarray*}
 r=T-\tau(T)=\sum_{i=1}^n (T(e_i)\otimes e_i^*-e_i^*\otimes T(e_i).
 \end{eqnarray*}
 Note that
 \begin{eqnarray*}
 ((N+\beta^*)\otimes \id)(r)=\sum\limits_{i=1}^{n}(N(T(e_i))\otimes e_i^*-\beta^*(e_i^*)\otimes T(e_i)),
 \end{eqnarray*}
 \begin{eqnarray*}
 (\id\otimes(S+\alpha^*))(r)=\sum\limits_{i=1}^{n}(T(e_i)\otimes \alpha(e_i^*)-e_i^*\otimes S(T(e_i))).
 \end{eqnarray*}
 Further,
 \begin{eqnarray*}
 \sum_{i=1}^n \beta^*(e_i^*)\otimes T(e_i)
 &=&\sum_{i=1}^n \sum_{j=1}^n \langle\beta^*(e_i^*),e_j\rangle e_j^*\otimes T(e_i)\\
 &=&\sum_{j=1}^n e_j^*\otimes  \sum_{i=1}^n \langle e_i^*,\beta(e_j)\rangle T(e_i)\\
 &=&\sum_{i=1}^n e_i^*\otimes T(\sum_{i=1}^n\langle\beta(e_i),e_j^*\rangle e_j)=\sum_{i=1}^n e_i^*\otimes T(\beta(e_i)),
 \end{eqnarray*}
 and similarly, $\sum_{i=1}^n T(e_i)\otimes \alpha(e_i^*)=\sum_{i=1}^n T(\alpha(e_i))\otimes e_i^*$. Therefore $((N+\beta^*)\otimes \id)(r)=(\id\otimes(S+\alpha^*))(r)$ if and only if $N\circ T=T\circ \alpha$ and $S\circ T=T\circ \beta$. Hence the conclusion follows.
 
 \ref{it:r6} It follows from \ref{it:xn} and Theorem \mref{thm:fv}.
 \end{proof}

\end{document}